\documentclass[11pt,letterpaper]{amsart}
\usepackage[plainpages=false]{hyperref}
\usepackage{amsfonts,latexsym,rawfonts,amsmath,amssymb,amsthm, mathrsfs, lscape, calligra}
\usepackage{verbatim}
\usepackage{enumerate}
\usepackage{centernot}
\usepackage{mathtools}
\usepackage{tikz}
\usepackage{tikz-cd}

\usepackage[all]{xy}
\usepackage{graphicx,psfrag}

\usepackage{array, tabularx, color}

\usepackage{setspace}

\newtheorem{thm}{Theorem}[section]
\newtheorem{cor}[thm]{Corollary}
\newtheorem{lem}[thm]{Lemma}

\newtheorem{prop}[thm]{Proposition}

\theoremstyle{remark}
\newtheorem{rmk}[thm]{Remark}

\theoremstyle{definition}
\newtheorem{defi}[thm]{Definition}

\newcommand{\CBbb}{\mathbb C}

\newcommand{\NBbb}{\mathbb N}

\newcommand{\RBbb}{\mathbb R}

\newcommand{\ZBbb}{\mathbb Z}

\newcommand{\Ecal}{\mathcal E}
\newcommand{\Fcal}{\mathcal F}

\newcommand{\Kcal}{\mathcal K}
\newcommand{\Lcal}{\mathcal L}
\newcommand{\Mcal}{\mathcal M}

\newcommand{\Ocal}{\mathcal O}

\newcommand{\GL}{\mathsf{GL}}

\newcommand{\SU}{\mathsf{SU}}

\DeclareMathOperator{\End}{End}
\DeclareMathOperator{\Hom}{Hom}

\DeclareMathOperator{\id}{id}

\DeclareMathOperator{\rank}{rank}

\DeclareMathOperator{\dvol}{dvol}

\DeclareMathOperator{\ad}{ad}

\DeclareMathOperator{\tr}{tr}

\DeclareMathOperator{\Sing}{Sing}

\newcommand{\dbar}{\bar\partial}

\numberwithin{equation}{section}

\makeatletter
\@namedef{subjclassname@2020}{\textup{2020} Mathematics Subject Classification}
\makeatother

\begin{document}
\title[Vafa-Witten equations]{On Vafa-Witten equations over K\"ahler manifolds}
\author[Chen]{Xuemiao Chen}

\begin{abstract}
In this paper, we study the analytic properties of solutions to the Vafa-Witten equation over a compact K\"ahler manifold. Simple obstructions to the existence of nontrivial solutions are identified. The gauge theoretical compactness for the $\CBbb^*$ invariant locus of the moduli space is shown to behave similarly as the Hermitian-Yang-Mills connections. More generally, this holds for solutions with uniformly bounded spectral covers such as nilpotent solutions. When spectral covers are unbounded, we manage to take limits of the renormalized Higgs fields which are intrinsically characterized by the convergence of the associated spectral covers. This gives a simpler proof for Taubes' results on rank two solutions over K\"ahler surfaces together with a new complex geometric interpretation. The moduli space of $\SU(2)$ monopoles and some related examples are also discussed in the final section.
\end{abstract}

\address{Department of pure mathematics, University of Waterloo, Ontario, Canada, N2L 3G1}\email{x67chen@uwaterloo.ca}
\maketitle
\tableofcontents
\thispagestyle{empty}

\bibliographystyle{amsplain}

\section{Introduction}
The Vafa-Witten equation is on a pair $(A, \phi)$ where $A$ is a unitary connection on a unitary bundle $(E,H)$ over a compact K\"ahler manifold $(X,\omega)$ and $\phi$ is a smooth section of $\End(E) \otimes K_X$ satisfying 
\begin{itemize}
\item[(1)] $F_A^{0,2}=0$;
\item[(2)] $\sqrt{-1} \Lambda_{\omega} F_{A}+[\phi;\phi]=\mu\id$;
\item[(3)] $\dbar_A \phi=0$.
\end{itemize}
Here $K_X$ denotes the canonical line bundle of $X$ and locally if we write $\phi=M \sigma$ for some smooth $n$-form $\sigma$, then 
$$
[\phi;\phi]=[M,M^*] |\sigma|^2.
$$
In literature, the Vafa-Witten equation is a special case of the Higgs equations (\cite{LubkeTeleman:06}). We call it the Vafa-Witten equation due to that when $\dim_{\RBbb} X=4$ it is the reduction of the Vafa-Witten equation over Riemannian four manifold (\cite{VafaWitten94}). Here we mention a few important and classical special cases. When $\phi=0$, the solutions are the Hermitian-Yang-Mills connections which correspond to slope polystable vector bundles by the celebrated Donaldson-Uhlenbeck-Yau theorem (\cite{Donaldson:87a} \cite{UhlenbeckYau:86}). When $\dim_{\RBbb} X=2$, i.e., over the Riemann surfaces, this corresponds to the Hitchin equation and the space of solutions corresponds to the space of polystable Higgs bundles (\cite{Hitchin87}). In general, the solutions of the Vafa-Witten equation correspond to the slope polystable Higgs bundles $(\Ecal, \phi)$ (\cite{LubkeTeleman:06} \cite{AG:03} \cite{Lin:89}).

Now we recall some background and motivations. The Vafa-Witten equation was originally proposed and studied by Vafa and Witten as a search of evidence for S-Duality (\cite{VafaWitten94}). The solutions are given by pairs $(A,a)$ where $A$ is a unitary connection on a principal $\SU(2)$-bundle $P$ over a Riemannian four manifold $(X,g)$ and $a$ is a  section of $\Lambda^+ \otimes \ad_{P}$ that satisfy
\begin{enumerate}
\item $d_Aa=0$;
\item $F_A^+=\frac{1}{2}[a;a]$ \footnote{The original version proposed by Vafa and Witten involves one more term and we refer the readers to \cite[Page 6]{Taubes17} for related discussions.}.
\end{enumerate}
When $(X,g)$ is compact K\"ahler, the equation is reduced as above (\cite{Mares:10}).  Assuming uniform $L^2$ bounds on Higgs fields, it is known in various situations (\cite{He20} \cite{Mares:10} \cite{Tanaka:13}) that the Uhlenbeck compactness can apply.  A very intriguing feature in general is that there \emph{does not} exist such a priori $L^2$ norm bound. For this, the sequential compactness for the solutions over general Riemannian four manifolds has been considered by Taubes (\cite{Taubes17}), where he managed to take the limit of the section part by looking at the renormalized section $\frac{a}{\|a\|_{L^2}}$ and produce a notion of $\ZBbb_2$ harmonic 2-form when $\|a\|_{L^2}$ becomes unbounded in a sequence. One of the main motivations for this paper is to understand such phenomenon in the complex analytic setting. 

Recently, by purely algebraic geometric methods, the moduli space of stable Higgs pairs over a projective surface has been used to define the Vafa-Witten invariants by Tanaka and Thomas (\cite{ThomasTanaka17, ThomasTanaka19}). The moduli space admits a natural $\CBbb^*$ action which comes from rescaling the Higgs fields. The key ingredient in this approach is the use of the $\CBbb^*$ invariant locus of the moduli space. A simple linear algebra lemma shows the $\CBbb^*$ invariant locus lies in the so-called nilpotent cone. It consists of two parts. One is the so-called monopoles corresponding to nonvanishing Higgs fields while the other is the classical Hermitian-Yang-Mills connections. An observation made in this paper is that even over general K\"ahler manifolds, one can control the curvature for the nilpotent solutions to the Vafa-Witten equations. Thus the known compactness results about Hermitian-Yang-Mills connections can be applied. More generally, similar statements hold when we have uniform control of the associated spectral covers. Furthermore, the moduli space of rank two monopoles is very simple which is compact and disjoint from the compactified moduli space of Hermitian-Yang-Mills connections.  This can be potentially used to study the Vafa-Witten invariants over K\"ahler surfaces using gauge theory and give an analytic interpretation of such over projective surfaces, while over general K\"ahler manifolds, this adds more interesting phenomena to the study of Hermitian-Yang-Mills connections. 

\subsection*{Main results} Now we briefly sketch the results proved in this paper. To better explain the relation with Taubes' notion of $\ZBbb_2$ harmonic $2$-forms over K\"ahler surfaces, we gradually increase the generality of discussions. 

We first observe some consequences of the Weitzenb\"ock formula by using an auxiliary Hermitian-Yang-Mills metric $H_{K_X}$ on $K_X$ instead of the naturally induced metric from the K\"ahler metric (see Section \ref{WeinzenbockFormula}). This gives the following restrictions on the base manifold.  
\begin{prop}
Suppose there exists a solution $(A,\phi)$ with $\phi\neq 0$ to the Vafa-Witten equation. Then 
$$
\deg(K_X) \geq 0
$$   
If $\deg K_X=0$ holds, then
$$
[\phi;\phi]=0, \nabla \phi=0
$$
and 
$$c_1(X)=0,$$ 
i.e., $X$ is a Calabi-Yau manifold \footnote{The K\"ahler metric $\omega$ is not necessarily a Calabi-Yau metric}. Here $\nabla$ denote the connection on $End(E)\otimes K_X$ induced by $A$ and the Chern connection on $K_X$ given by $H_{K_X}$.
\end{prop}

\begin{rmk}
    \begin{itemize}
        \item When $\dim_\CBbb X=1$, this is known (\cite{Hitchin87}).
        \item When $\deg K_X=0$, it is known that there are no monopole solutions over projective complex surfaces (\cite{ThomasTanaka17, DonaldsonThomas:98}) and $[\phi;\phi]=0$ over K\"ahler manifolds (\cite{MMS:20}).
    \end{itemize}
\end{rmk}

In particular, this gives 
\begin{cor}
Over a Calabi-Yau manifold, the solutions $(A,\phi)$ to the Vafa-Witten equation consist of a Hermitian-Yang-Mills connection $A$ and a Higgs field $\phi$ which locally can be written as $\phi=M \nu$ where $\nu$ is a nonzero holomorphic $n$-form so that 
\begin{itemize}
\item $\nabla_A M =0, |M|=1$;
\item $\nabla_{H_K} \nu =0$
\end{itemize}  
where $\nabla_{H_K}$ is induced by a flat  metric on $K_X$. In particular, by putting local choices of all such $\nu$ together, this gives a notion of $\ZBbb_k$ holomorphic $n$-form for some $k$ (see Corollary \ref{formCY}). 
\end{cor}

We obtain analytic results on nilpotent solutions to the Vafa-Witten equation and more generally solutions with uniformly bounded spectral covers . Here $\phi$ if we write $\phi=Mdz_1 \wedge \cdots dz_n$, then the spectral cover is essentially given by the eigenvalues of $M$ (see definition \ref{Defi2.5}); if $M$ is nilpotent, we call $(A,\phi)$ nilpotent. 

\begin{thm}\label{NilpotentCompactness}
There exists a uniform $C^0$ bound on the Higgs fields of nilpotent solutions including monopole solutions to the Vafa-Witten equation over compact K\"ahler manifolds. In particular, the Uhlenbeck compactness can be applied; more generally this holds for a sequence of solutions to the Vafa-Witten equation with uniformly bounded spectral covers. More precisely, given any sequence of such solutions $(A_i, \phi_i)$ to the Vafa-Witten equation, passing to a subsequence, up to gauge transforms, $(A_i, \rho_i=\frac{\phi_i}{\|\phi_i\|})$ converges locally smoothly to $(A_\infty, \rho_\infty)$ over $X\setminus Z$ where 
$$
Z=\{x\in X: \lim_{r\rightarrow 0} \liminf_{i}\int_{B_x(r)} |F_{A_i}|^2 \dvol \geq \epsilon_0  \}
$$
is a codimension at least two subvariety of $X$ and $\epsilon_0$ is the regularity constant. Furthermore, $A_\infty$ defines a unique reflexive sheaf $\Ecal_\infty$ over $X$ and $\rho_\infty$ extends to a global section of $Hom(\Ecal_\infty, \Ecal_\infty) \otimes K_X$ over $X$; $Z$ is a codimension at least two subvariety that admits a decomposition 
$$
Z=\cup_k Z_k \cup \Sing(\Ecal_\infty)
$$
where $Z_k$ denotes the irreducible pure codimension two components of $Z
$ and $\Sing(\Ecal_\infty)$ denotes the codimension at least three locus where $\Ecal_\infty$ fails to be locally free. As a sequence of currents 
$$
\tr(F_{A_i} \wedge F_{A_i}) \rightarrow \tr(F_{A_\infty} \wedge F_{A_\infty})+ 8\pi^2 \sum m_k Z_k
$$
where $m_k \in \ZBbb_+$, i.e., for any compact supported $2n-4$-form $\Psi$, the following holds 
$$
\int \tr(F_{A_i} \wedge F_{A_i}) \wedge \Psi\rightarrow \int \tr(F_{A_\infty} \wedge F_{A_\infty}) \wedge \Psi+ 8\pi^2 \sum m_k \int_{Z_k} \Psi.
$$
\end{thm}

\begin{rmk}
The compactness here is known to be sharp for Hermitian-Yang-Mills connections in general (\cite{Tian:00,GSTW:18, ChenSun:19}), i.e., complex codimension two bubbling with appearance of essential singularities is a genuine phenomenon. Note in the Riemann surface case,  no bubbling happens since it is complex codimension two. Thus by only assuming bounds on the spectral cover in general, the compactness above describes the general case. But in some special cases, one can say more (see section \ref{Monopoles} for rank two monopoles).
\end{rmk}

Below we assume that the spectral covers associated to $(A_i, \phi_i)$  are not uniformly bounded. By passing to a subsequence, this is equivalent to 
$$
\lim_i \|\phi_i\|_{L^2} = \infty
$$
(see Corollary \ref{L2ImpliesC0}). Denote $\rho_i=\frac{\phi_i}{\|\phi_i\|_{L^2}}$.

First, to explain the connection with Taubes' notion of $\ZBbb_2$ harmonic $2$-forms, we assume $E$ is a rank two bundle and consider solutions with trace free Higgs fields.

\begin{thm}\label{Rank=2}
Given a sequence of rank two solutions $(A_i,\phi_i)$ to the Vafa-Witten equation with $\tr \phi_i=0$ and $\lim_i \|\phi_i\|=\infty$, there exists a $\ZBbb_2$ holomorphic $n$-form $(L, \nu, Z)$ associated to the spectral cover defined by $b_\infty$ (see Lemma \ref{Lemma6.4})  so that  there exists a sequence of isometric embeddings
$$
\sigma_i: L|_{X\setminus Z} \rightarrow \End(\Ecal_i)|_{X\setminus Z}  
$$
satisfying
$$
\rho_i - \sigma_i \nu   \xrightarrow{C^0(Y)} 0,
$$
$$
|\nabla_{A_i} \sigma_i| \xrightarrow{L^2(Y)}0,
$$
and 
$$
\nabla_{A_i}(\rho_i - \sigma_i \nu)  \xrightarrow{L^2(Y)} 0
$$
over any compact subset $Y\subset X\setminus Z$.
\end{thm}

\begin{rmk}
\begin{itemize}
    \item When $\dim_{\CBbb} X=2$ and $(E,H)$ has structure group $\SU(2)$, this has been obtained by Taubes (\cite{Taubes17}). The results here give a simpler proof of Taubes' results together with a complex geometric interpretation of the limiting data (see Corollary \ref{TaubesZ2}).
    \item When $\dim_{\CBbb }X=1$, this is also known (\cite{MSWW:16}, \cite{Mochizuki:16}).
\end{itemize}
 
\end{rmk}
 
\emph{In principle} (see Remark \ref{rmk1.10}), we give an intrinsic description of the role of the spectral cover by using the natural torsion free sheaves associated to the spectral covers. More precisely, over the total space of the canonical bundle $K_X$, there exists a tautological line bundle $\Kcal\cong \pi^* K_X$ which has a global tautological section $\tau$ over $K_X$. Given any spectral cover $\pi: X_b \rightarrow X$ associated to 
$$b\in \oplus_{i=1}^r H^0(X, K_X^{\otimes i})$$ 
(see Definition \ref{Defi2.5}), we denote 
$$
\Kcal^b=\Kcal|_{X_b}.
$$
Then $\pi_* \Kcal^b$ is a locally free sheaf of rank equal to $r$ away from the discriminant locus $\Delta_b$ and it has a tautological section 
$$
\tau_b=\pi_*(\tau|_{X_b}).
$$
Furthermore, it has a natural Hermitian metric induced by $K_X$ away from $\Delta_{b}$. 

We denote $X_{b_\infty}$ as the limit of the spectral cover $X_{b(\rho_i)}$ where $\rho_i = \frac{\phi_i}{\|\phi_i\|_{L^2}}$. Then in general, we have the following intrinsic form of the convergence of the renormalized Higgs fields away from the discriminant locus $\Delta_{b_\infty}$ of the spectral cover $X_{b_\infty}$.

\begin{thm}\label{Main}
Assume $\lim_i \|\phi_i\|=\infty$. By passing to a subsequence, there exists a sequence of isometric embeddings
$$
\sigma_i: \pi_*\Kcal^{b_\infty}|_{X\setminus \Delta_{b_\infty}} \rightarrow \End(\Ecal_i) \otimes K_X|_{X\setminus \Delta_{b_\infty}} 
$$
so that 
$$
\sigma_i \tau_{b_\infty}-\rho_i \xrightarrow{C^0(Y)} 0,
$$
$$
\nabla_i \sigma_i \xrightarrow{L^2(Y)} 0,
$$
and 
$$
\nabla_i(\sigma_i \tau_{b_\infty}-\rho_i )\xrightarrow{L^2(Y)}0
$$
over any fixed compact subset $Y\subset X\setminus \Delta_{b_\infty}$.
\end{thm}

\begin{rmk}\label{rmk1.10}
We emphasize the description here is in principle as we mentioned above. This is due to that the discriminant locus $\Delta_b$ can be the whole set $X$ in general. Note the nilpotent solutions including monopoles and Hermitian-Yang-Mills connection gives $\Delta_b=X$, but the Uhlenbeck compactness can still be applied. 
\end{rmk}

In the last section, we study the space of rank two monopoles with trace free Higgs fields and some examples. It turns out that in general, given a monopole, the connection has to be reducible (see Proposition \ref{Reducible}). In particular, an $\SU(2)$ monopole is given by a connection on a line bundle $L$ together with a section of $K_X \otimes L^{-2}$, thus the moduli space of $\SU(2)$ monopoles is compact. We also give an example that even on the trivial bundle, the moduli space of stable Higgs pairs is rich over K\"ahler manifolds with two linearly independent holomorphic forms of top degree. 

\subsection*{Notation} Give two quantities $Q_1$ and $Q_2$ which are usually norms of matrix valued functions, we use 
\begin{itemize}
\item $Q_1\lesssim (\gtrsim) Q_2$ to denote $Q_1 \leq (\geq) C Q_2$ for some constant $C>0$ independent of $Q_1$ and $Q_2$ 
\item $Q_1 \sim Q_2$ to denote $C^{-1} Q_2 \leq Q_1 \leq C Q_2$ for some constant $C>0$ independent of $Q_1$ and $Q_2$
\end{itemize}

\subsection*{Acknowledgment} The author would like to thank Siqi He for helpful discussions on related topics and Ruxandra Moraru for pointing out several references. He would also like to thank the anonymous referee for carefully reading the paper, pointing out related references, and many valuable questions and suggestions which greatly improved the presentation of the paper. This work is partially supported by NSERC and the ECR supplement.

\section{Preliminaries}
\subsection{Linear algebra at one point}
Let $V$ be a complex vector space of dimension $m$. Consider the adjoint action  $\GL(V) \curvearrowright \End(V)$ and the GIT quotient $\End(V)//\GL(V)$. As a space, this parametrizes the space of $\GL(V)$-orbits in $\End(V)$ with two orbits  identified if their closures intersect nontrivially. Denote the natural projection by $p: \End(V)/\GL(V) \rightarrow \End(V)//\GL(V).$ The following is well-known and instructive for what type of generic matrices one might consider in general 

\begin{prop}
For any $y\in \End(V)//\GL(V)$, $p^{-1}(y)$ contains a unique semisimple orbit , i.e., the orbit contains a diagonalizable representative, and a unique regular orbit , i.e., the orbit contains a representative so that each eigenvalue has exactly one Jordan block. Furthermore,
$$
\End(V)//\GL(V)\cong \CBbb^m
$$
and the identification is given by the invariant functions as those $b_i(\phi)$ satisfying 
$$
\det(\lambda\id- \phi)=\lambda^m+b_1(\phi)\lambda^{m-1}+\cdots b_m(\phi).
$$
\end{prop}

We also need a well known simple linear algebra lemma which we include a proof for completeness
\begin{lem}\label{MatrixNorm}
The following holds for any $m\times m$ matrices
\begin{enumerate}
    \item $|M|^2 \sim |[M,M^*]|+|b_1(M)|^2 + \cdots |b_m(M)|^{\frac{2}{m}}$;
    \item $0\leq |M|^2-\sum_i |\lambda_i|^2\lesssim |[M,M^*]|$ where $\{\lambda_i\}_i$ denote the set of eigenvalues counted with multiplicities.
\end{enumerate}
\end{lem}

\begin{proof}
For (1), we first show 
$$
|M|^2 \lesssim |[M,M^*]|+|b_1(M)|^2 + \cdots |b_m(M)|^{\frac{2}{m}}.
$$
Otherwise, there exists a sequence of matrices $M_k$ with $|M_k|=1$, but 
$$
|[M_k,M_k^*]|+|b_1(M_k)|^2 + \cdots |b_n(M_k)|^{\frac{2}{n}} \leq \frac{1}{k}.
$$
Passing to a subsequence, we can assume $M_k \rightarrow M$ with $|M|=1$. Also, we obtain
$$|[M,M^*]|+|b_1(M)|^2 + \cdots |b_m(M)|^{\frac{2}{m}}=0.$$
which implies $M=0$. This is a contradiction. The other direction is trivial. 

For (2), we only need to show the second inequality. Otherwise, there exists a sequence of matrices $M_k$ so that 
$$
k|[M_k, M_k^*]|\leq |M_k|^2-\sum_i |\lambda_i^k|^2.
$$
By doing unitary transforms, we can assume $M_k=D_k+U_k$ where $D_k$ is diagonal and $U_k$ is strictly upper triangular and nonzero. We normalize $|U_k|=1$. Then 
$$
|[U_k,U_k^*]+[D_k, U_k^*]+[D_k^*,U_k]|=|[M_k, M_k^*]| \leq \frac{1}{k}.
$$
Since $[D_k, U_k^*]+[D_k^*,U_k]=[D_k, U_k^*]-[D_k,U_k^*]^*$ is skew symmetric and $[U_k,U_k^*]$ is symmetric, they are orthogonal. Thus the above implies 
$$
|[U_k,U_k^*]|^2 \leq |[U_k,U_k^*]+[D_k, U_k^*]+[D_k^*,U_k]|^2 \leq \frac{1}{k^2}.
$$
Passing to a subsequence, we can assume $U_k \rightarrow U$ which is strictly upper triangular, $|U|=1$ and $[U,U^*]=0.$ This is a contradiction. 
\end{proof}

\subsection{Stability} 
We will use the following standard notion of stability for a Higgs pair $(\Ecal, \phi)$ where $\Ecal$ is a holomorphic vector bundle over a K\"ahler manifold $(X,\omega)$ together with a holomorphic section $\phi\in H^0(X, \End(\Ecal) \otimes K_X)$ (\cite{LubkeTeleman:95}).
\begin{defi}
$(\Ecal, \phi)$ is called stable (semistable) if for any subsheaf $\Fcal \subset \Ecal$ with $\phi(\Fcal) \subset \Fcal\otimes K_X$ and $0<\rank \Fcal < \rank \Ecal$,  the following holds 
$$
\mu(\Fcal)\equiv \frac{\int_X c_1(\Fcal) \wedge \omega^{n-1}}{\rank \Fcal}< \mu(\Ecal)\equiv \frac{\int_X c_1(\Ecal) \wedge \omega^{n-1}}{\rank \Ecal}
$$
\end{defi}

\begin{defi}\label{Defi2.5}
Given a Higgs pair $(\Ecal, \phi)$, the spectral cover $X_{b(\phi)}$ associated to $(\Ecal, \phi)$ is defined as
$$
X_{b(\phi)}=\Psi^{-1}(0) \subset K_X
$$
where 
$$
\Psi: K_X \rightarrow K_X^{\otimes r}, \lambda \mapsto \det(\lambda \id_{\Ecal}-\phi)=\lambda^r+b_{r-1}(\phi) \lambda^{r-1} + \cdots b_r(\phi).
$$
and 
$$
b(\phi)=(b_1(\phi), \cdots b_r(\phi)). 
$$
Here $r=\rank \Ecal$. The natural projection $\pi: X_{b(\phi)} \rightarrow X$ is a covering branched along the discriminant locus 
$$
\Delta_{b(\phi)}=(P(b_1(\phi),\cdots b_r(\phi))=0) \subset X
$$
where $P$ denotes the discriminant polynomial for 
$$
t^r+b_{r-1} t^{r-1} + \cdots b_r=0
$$
and $P(b_1(\phi),\cdots b_n(\phi))\in H^0(X, K_X^{\otimes r})$.
\end{defi}

\begin{rmk}
\begin{itemize}
    \item For purpose later, we choose $b(\phi)$ to index the spectral cover $X_{b(\phi)}$ because it only needs data from $b(\phi)$ rather than all the information about $\phi$. Indeed, given any $b\in \oplus_{k=1}^r H^0(X, K_X^{\oplus k})$, we can define 
$$
X_b=\{\lambda \in K_X: \lambda^r + b_1 \lambda^{r-1} + \cdots b_r=0\}
$$
as above which is a covering of $X$ branched along the discriminant locus $\Delta_{b}$ similar as above. We will still call $X_b$ the spectral cover associated to $b$. Assuming $b_1=\cdots b_{r-1}=0$, $X_b$ is the cyclic cover by taking $r$-th root of $-b_r$ in $K_X$.
    \item We also mention that over a projective manifold $X$, there exists a natural one-to-one correspondence between the space of Higgs sheaves $(\Ecal, \phi)$ where $\phi\in H^0(X, \End(\Ecal) \otimes K_X)$ and the space of compactly supported coherent sheaves $\hat{\Ecal}$ over $K_X$. In particular, $supp(\hat{\Ecal})=X_{b(\phi)}$ (\cite[Proposition $2.2$]{ThomasTanaka17}).
\end{itemize}

\end{rmk}
In the following, we fix a smooth bundle $E$ over $X$. Denote
$$
\Mcal_{Higgs}:=\{(\Ecal,\phi): \Ecal \text{ has $E$ as the underlying smooth bundle }\}/\sim
$$
where $(\Ecal,\phi) \sim (\Ecal',\phi')$ if and only if there exists an isomorphism $f$ so that the following diagram commutes
\[\begin{tikzcd}
	\Ecal & {\Ecal\otimes K_X} \\
	\Ecal' & {\Ecal' \otimes K_X.} \\
	\arrow["\phi", from=1-1, to=1-2]
	\arrow["f"', from=1-1, to=2-1]
	\arrow["\phi'"', from=2-1, to=2-2]
	\arrow["f", from=1-2, to=2-2]
\end{tikzcd}\]
We also denote the subspace of stable Higgs pairs as $\Mcal_{Higgs}^s \subset \Mcal_{Higgs}$. 

\begin{defi}
The Hitchin map is defined as
$$
h: \Mcal_{Higgs}\rightarrow \oplus_{k=0}^{\rank \Ecal} H^0(S, K_X^{\otimes k}), (\Ecal, \phi) \rightarrow (b_1(\phi), \cdots, b_n(\phi))
$$
and the fiber $h^{-1}(0)$ is called the nilpotent cone.  
\end{defi}

\subsection{Vafa-Witten equation} Given a smooth unitary bundle $(E,H)$ over a compact K\"ahler manifold $(X,\omega)$, a solution to the Vafa-Witten equation is a pair $(A, \phi)$ satisfying 
\begin{itemize}
\item[(1)] $F_A^{0,2}=0$;
\item[(2)]  $\sqrt{-1} \Lambda_{\omega} F_{A}+[\phi;\phi]=\mu\id$;
\item[(3)] $\dbar_A \phi=0$,
\end{itemize}
where $A$ is a unitary connection on $(E,H)$, $\phi$ is a section of $\End(E) \otimes K_X$ and 
$$
\mu=\frac{2\pi \int_X c_1(E) \wedge \frac{\omega^{n-1}}{(n-1)!}}{\int_X \frac{\omega^n}{n!}}
$$ 
which is usually called the Hermitian-Einstein constant. Here locally if we write $\phi=M \sigma$ for some holomorphic $n$-form $\sigma$, then 
$$
[\phi;\phi]=[M,M^*] |\sigma|^2.
$$
\begin{defi}
A solution $(A,\phi)$ to the Vafa-Witten equation is called irreducible if it can not be written as a direct sum of nontrivial solutions to the Vafa-Witten equation.
\end{defi}

We denote by $\Mcal^*$ as the space of irreducible solutions to the Vafa-Witten equation mod gauge equivalence. Here two solutions $(A,\phi)$ and $(A', \phi')$ are called equivalent if there exists a smooth unitary isomorphism $f:E \rightarrow E$ so that 
$$
f \circ \nabla_A \circ f^{-1}=\nabla_{A'}, f \phi f^{-1} = \phi'.
$$
Then there exists a real analytic isomorphism $\Phi: \Mcal^{s}_{Higgs}\rightarrow \Mcal^*$ (\cite{LubkeTeleman:06})). Later in this paper, we will study the limiting behavior of a sequence of solutions to the Vafa-Witten equation. For this, we need
\begin{lem}
Given a solution $(\Ecal, \phi)$ to the Vafa-Witten equation, the following holds 
$$
C^{-1}(\|[\phi;\phi]\|_{L^2}+\mu)\leq \|F_A\|_{L^2} \leq C(\|[\phi;\phi]\|_{L^2}+\mu)
$$
for some dimensional constant $C$.
\end{lem}

\begin{proof}
Since by definition
$$
F_A = \Lambda_{\omega} F_A.\omega + F_A^{0},
$$
the Hodge-Riemann property implies 
$$
|F_A^0|^2 \sim -\frac{\tr(F_A^0 \wedge F_A^0) \wedge \frac{\omega^{n-2}}{(n-2)!}}{\frac{\omega^n}{n!}}
$$
In particular, we obtain 
$$
\begin{aligned}
|F_A|^2 &\sim |\Lambda_{\omega} F_A|^2 -\frac{\tr(F_A^0 \wedge F_A^0) \wedge \frac{\omega^{n-2}}{(n-2)!}}{\frac{\omega^n}{n!}} \\
&\sim  |\Lambda_{\omega} F_A|^2 - \frac{\tr(F_A \wedge F_A) \wedge \frac{\omega^{n-2}}{(n-2)!}}{\frac{\omega^n}{n!}} \\
& \sim |[\phi;\phi]|^2-\frac{\tr(F_A \wedge F_A) \wedge \frac{\omega^{n-2}}{(n-2)!}}{\frac{\omega^n}{n!}}+\mu^2. 
\end{aligned}
$$
The conclusion follows.
\end{proof}

\begin{cor}
Given a sequence of solutions $(A_i, \phi_i)$ to the Vafa-Witten equation on a fixed unitary bundle, $\|F_{A_i}\|_{L^2}$ is uniformly bounded if and only if $\|[\phi_i;\phi_i]\|_{L^2}$ is uniformly bounded. 
\end{cor}

\begin{rmk}
In general, it is known that such a bound does not exist globally. This is well-known in the Riemann surface case and similar examples exist in higher dimension too. For example, take a stable pair $(\Ecal, \phi)$ where $\phi$ is not nilpotent (see Corollary \ref{Cor8.7}), then $(\Ecal, t\phi)$ as $t \rightarrow \infty$ give such a family. This makes the compactification problem for the Vafa-Witten equation more subtle than the Hermitian-Yang-Mills case in general.
\end{rmk}

\subsection{Monopoles} We start with the following definition
\begin{defi}\label{CStarInvariant}
A Higgs pair $(\Ecal, \phi)$ is called $\CBbb^*$ invariant if for any $t\in \CBbb^*$, there exists an isomorphism $f: \Ecal \rightarrow \Ecal$ so that the following diagram commutes
\[\begin{tikzcd}
	\Ecal & {\Ecal\otimes K_X} \\
	\Ecal & {\Ecal \otimes K_X} \\
	\arrow["\phi", from=1-1, to=1-2]
	\arrow["f"', from=1-1, to=2-1]
	\arrow["t\phi"', from=2-1, to=2-2]
	\arrow["f", from=1-2, to=2-2]
\end{tikzcd}\]
which is equivalent to that $(\Ecal,\phi)$ gives a fixed point in $\Mcal_{Higgs}$ under the $\CBbb^*$ action.
\end{defi}

We need the following simple observation
\begin{lem}
Given a $\CBbb^*$ invariant Higgs pair $(\Ecal, \phi)$, it must be nilpotent. 
\end{lem}

\begin{proof}
By definition, for any $t$, there exists some isomorphism $f$ so that
$$
f^{-1}  \circ t\phi \circ f = \phi.
$$
which implies for any $k\geq 1$
$$
b_k(\phi)=b_k(t\phi)
$$
for any $t\in \CBbb^*$. For a fixed point $x\in X$, since $b_i(t\phi)(x)$ is a polynomial in $t$ and $b_k(0)(x)=0$, we must have 
$$
b_k(\phi)\equiv 0
$$
for any $k\geq 1$. This implies $\phi$  is nilpotent.
\end{proof}

More generally, by repeating the argument of \cite[Lemma 4.1]{Simpson:92}, we can prove
\begin{prop}\label{prop2.16}
A Higgs pair $(\Ecal, \phi)$ is $\CBbb^*$ invariant if and only if there exists a decomposition 
$$
\Ecal=\oplus_{i,j} \Ecal_{i,j}
$$
where $\Ecal_{i,j}$ are holomorphic sub-bundles of $\Ecal$ so that 
$$
\phi: \Ecal_{i,j} \rightarrow  \Ecal_{i-1, j+1} \otimes K_X.
$$
\end{prop}

\begin{proof}
Given $\CBbb^*$ invariant Higgs pair $(\Ecal, \phi)$, we can choose $t\in \CBbb^*$ which is not a root of unity. Then by definition, there exists an isomorphism $f$ of $\Ecal$ so that 
$$
f\circ \phi= t \phi \circ f.
$$ 
Since $b_k(f)$ is a constant, the eigenvalues of $f_t$ are constants. For any eigenvalue $\lambda$ of $f$, we denote 
$$
\Ecal_\lambda=\ker((\lambda \id-f)^{\rank \Ecal}).
$$
Since $f\circ \phi= t \phi \circ f$, $\phi$ maps $\Ecal_\lambda$ to $\Ecal_{t\lambda}$. Thus we can further write 
$$
\Ecal=\oplus_{i=1}^s (\Ecal_{\lambda_i} \oplus \cdots \Ecal_{t^{n_i}\lambda_i})
$$
where 
\begin{itemize}
\item $\Ecal_{t^j\lambda_i}=\ker((t^j\lambda_i\id-f)^{\rank \Ecal})$, $t^{-1} \lambda_i$ and $t^{n_i+1} \lambda_i$ are not eigenvalues for $f$;
\item $t^j \lambda_i \neq t^{j'} \lambda_{i'}$ for $(i,j) \neq (i', j')$.
\end{itemize}
Furthermore, $\phi$ maps $\Ecal_{t^j\lambda_{i}}$ to $\Ecal_{t^{j+1} \lambda_i}$. Now we can index the decomposition appropriately to get the desired decomposition. For the other direction, given such $\phi$, for any $t$, we can construct $f$ by scaling each $\Ecal_{i,j}$ with $t^i$.
\end{proof}

There are two important classes of solutions to the Vafa-Witten equations. The first one is when $\phi=0$, it recovers the Hermitian-Yang-Mills connections and also referred as instantons when $\dim_{\RBbb} X =4$. The second type is the so-called monopoles 

\begin{defi}
An irreducible solution $(A, \phi)$ to the Vafa-Witten equation is called a monopole if $((E, \dbar_A), \phi)$ is $\CBbb^*$ invariant and $\phi\neq 0$. We also call the corresponding Higgs bundle $(\Ecal, \phi)$ a monopole. 
\end{defi}

We note the following
\begin{prop}\label{Reducible}
Suppose $(A, \phi)$ is a monopole. For any $1\neq t\in S^1$, there exists an isomorphism 
$$f_t: (E, \dbar_A) \rightarrow (E, \dbar_A)$$ 
so that 
$$
t\phi \circ f_t=f_t \circ \phi
$$
and
$$
\nabla_A f_t =0.
$$
\end{prop}

\begin{proof}
By definition, for any $t$, there exists a unitary isomorphism 
$$f_t: \Ecal\equiv (E, \dbar_A) \rightarrow \Ecal$$
so that 
$$
f_t^{-1}\circ t\phi \circ f_t = \phi.
$$
Since $\phi \neq 0$, we obtain $f_t$ is not a scaling for any $t\neq 1$. By stability,  $(\Ecal, t\phi)$ also admits Hermitian metric $H_t$ so that the Chern connection $A_t$ together with $t\phi$ satisfies the Vafa-Witten equation. By uniqueness (\cite[Proposition $3.10$]{AG:03}), the metric has to be a multiple of the original background metric. By rescaling, we can assume the metrics are the same. Endow $\End(\Ecal)$ with the connection 
$$
\nabla_{AA_t}s=\nabla_A \circ s- s\circ \nabla_{A_t}
$$
which satisfies 
$$
\Lambda_{\omega} F_{AA_t}=\Lambda_{\omega} F_{A}-\Lambda_{\omega} F_{A_t}=0.
$$
since $t\in S^1$, $[t\phi;t\phi]=[\phi;\phi]$. Since $f_t$ is a holomorphic section of $\End(\Ecal)$, it must be parallel with respect to $\nabla_{AA_t}$, i.e., $A_t=f_t^*A$. Thus $(f_t^*A, f_t^*\phi)=(f_t^*A, t\phi)$ gives another solution to the Vafa-Witten equation corresponding to the Higgs pair $(\Ecal, t\phi)$ and we have 
$$
\Lambda_\omega (F_{f_t^*A}-F_{A})=0
$$
for any $t\in S^1$. This is equivalent to 
$$
\Lambda_{\omega}(\dbar_A (f_t^{-1} \partial_A f_t))=0
$$
and we can simplify it as 
$$
\Lambda_{\omega}(\dbar_A(\partial_A f_t))=0
$$
which by the K\"ahler identities implies 
$$
\partial_A^* \partial_A f_t=0
$$
thus 
$$
\partial_A f_t=0
$$
for any $t\in S^1$. Combined with $f_t$ being holomorphic, this implies $f_t$ is parallel. 
\end{proof}

\section{The Weitzenb\"ock formula with consequences}
Below we fix a solution $(A, \phi)$ to the Vafa-Witten equation over a compact K\"ahler manifold $(X,\omega)$. Instead of using the naturally induced connection on $K_X$ from the K\"ahler metric, we fix a Hermitian-Yang-Mills metric $H_K$ on $K_X$ and the curvature of the associated Chern connection $\nabla_{H_K}$ on $K_X$ satisfies
$$
\sqrt{-1} \Lambda_{\omega}F_{H_K}=2\pi \deg(K_X).
$$
Then we denote by $B$ the Chern connection on $\End(E) \otimes K_X$ induced by $A$ and $\nabla_{H_K}$. The norm of the section for $\phi$ will be computed by $H_{K}$ as 
$$
|\phi|=|M||\sigma|_{H_K}
$$
where locally $\phi=M\sigma$. 

\begin{prop}\label{WeinzenbockFormula}
The following holds 
$$\frac{1}{2}\nabla_{B}^* \nabla_{B} \phi + [[\phi; \phi], \phi]-2\pi\deg (K_X) \phi=0.$$
In particular, 
$$
\Delta_{\dbar} |\phi|^2 +|[\phi;\phi]|^2 +|\nabla_B \phi|^2 \sim 2\pi\deg(K_X) |\phi|^2
$$
and 
$$
\|\nabla_{B} \phi\|_{L^2}^2 + \|[\phi;\phi]\|_{L^2}^{2} \sim 2\pi \deg(K_X) \|\phi\|^2_{L^2}.
$$
\end{prop}

\begin{proof}
By the K\"ahler identities, 
\begin{equation}
0=\dbar_{B}^{*}  \dbar_{B} \phi=\frac{1}{2}\nabla_{B}^* \nabla_{B} \phi-  [\sqrt{-1}\Lambda F_{A}, \phi]-\sqrt{-1} \Lambda F_{H_K}\phi
\end{equation}
which combined with the Vafa-Witten equation implies 
$$
0=\dbar_{B}^{*} \dbar_{B} \phi=\frac{1}{2}\nabla_{B}^* \nabla_{B} \phi + \frac{1}{2} [[\phi; \phi], \phi]- 2\pi\deg (K_X)\phi.
$$
The first equality follows. For the second one, it follows from taking inner product of the equality above with $\phi$ and using that $\Delta_{\dbar}=\frac{1}{2} \Delta_d$. Here we get $\sim$ instead of equality because $([[\phi;\phi],\phi],\phi) \sim |[\phi;\phi]|^2$  where $|[\phi;\phi]|^2$ is computed using the K\"ahler metric on the form part. The last follows from integration of the equality above. 
\end{proof}
We note the following 

\begin{cor}\label{L2ImpliesC0}
    $\|\phi\|_{C^0}\lesssim \|\phi\|_{L^2}$.
\end{cor}

\begin{proof}
    Indeed, we obtain from Proposition \ref{WeinzenbockFormula} that
    $$
    \Delta_{\dbar_A}|\phi|^2 \lesssim |\phi|^2.
    $$
 which implies $\Delta_{\dbar_A}|\phi| \lesssim |\phi|  $ in the weak sense.   The conclusion now follows from the Moser iteration. 
\end{proof}

\begin{cor}
Given a solution $(A,\phi)$ to the Vafa-Witten equation over a K\"ahler manifold $(X,\omega)$, 
\begin{itemize}
\item suppose $\phi \neq 0$, then $\deg K_X \geq 0$;
\item suppose $[\phi;\phi]\neq 0$, then $\deg K_X>0$.
\end{itemize}
\end{cor}

Below we examine the case of $\deg K_X=0$.

\begin{cor}\label{deg=0}
Assume $\deg K_X=0$. Then any nontrivial solution $(A,\phi)$ to the Vafa-Witten equation must satisfy 
$$
[\phi;\phi]=0, \text{ and } \nabla_B \phi = 0.
$$
Furthermore, if $\phi \neq 0$, then $K_X^{\otimes k}$ is trivial for some $0<k\leq \rank \Ecal$. In particular, $X$ must be a Calabi-Yau manifold. 
\end{cor}
\begin{proof}
The first part follows directly from the Weitzenb\"ock formula above. Since $[\phi;\phi]=0$, $\phi$ can be diagonalized. Thus if $\phi \neq 0$, $b_k(\phi) \neq 0$ for some $0<k\leq \rank \Ecal$. Since $\deg K_X^{\otimes k}=0$, $b_k(\phi)$ must trivialize $K_X^{\otimes k}$, thus $c_1(K_X)=0$. The conclusion follows.
\end{proof}

In particular, we have 
\begin{cor}\label{prop6.1}
Locally $\phi=M \nu$ 
\begin{itemize}
\item $\nabla_A M=0$;
\item $\nabla_{H_K}\nu=0$.
\end{itemize}
where $\nu$ is a local holomorphic $(n,0)$ form and $M$ is a local section of $\End(\Ecal)$. 
\end{cor}

\begin{proof}
Since $H_K$ is a flat metric, locally we can always choose a holomorphic $(n,0)$ form $\nu\neq 0$ so that 
$$
\nabla_{H_K} \nu = 0.
$$
Then we can write $\phi=M\nu$. Since 
$$0=\nabla_A (M\nu)=\nabla_A(M) \nu + M \nabla_{H_K} \nu = (\nabla_A M)\nu$$
we obtain 
$$
\nabla_A M =0.
$$
The conclusion follows.
\end{proof}

Note given any matrix $M$, suppose $k$ is any integer so that $b_k(M)\neq 0$. Then the condition
$$
b_k(e^{i\theta}M)\in \RBbb_+
$$
has exactly $k$ solutions for $e^{i\theta}$. Given this, we have the following 
\begin{cor}\label{formCY}
A nontrivial Vafa-Witten solution $(A,\phi)$ over a Calabi-Yau manifold naturally gives rise to a $\ZBbb_k$ holomorphic $n$-form for some $k$, i.e., a global section of $K_X$ over $X$ defined up to a $\ZBbb_k$ action.
\end{cor}

\begin{proof}
Indeed, by Proposition \ref{prop6.1}, for any $x\in X$, locally over a neighborhood of $x$, we can write 
\begin{itemize}
\item $\nabla_A M_x=0$, $|M_x|=1$;
\item $\nabla_{H_K}\nu=0$.
\end{itemize}
where $\nu$ is a local holomorphic $(n,0)$ form and $M$ is a local section of $\End(\Ecal)$. Since $[\phi;\phi]=0$, $\phi$ is diagonalizable. Since $\phi \neq 0$, we obtain $b_k(\phi) \neq 0$ for some $k$. Now if we put all the local choices together $\cup_x (U_x, M_x)$ by requiring $M_x$ satisfy
$$
b_k(M_x)\in \RBbb^+
$$ 
different choices differ by $\ZBbb_k$ as noted above which gives a natural principal $\ZBbb_k$ bundle $P_{\ZBbb_k}$. Let $L^{-1}$ be the natural flat bundle associated to $P_{\ZBbb^k}$ and denote $L$ as the dual of $L^{-1}$. Since $M\nu$ is globally defined as $\rho$, we obtain the corresponding choice of local choices $\nu$ glued together to be a holomorphic section of $L\otimes K_X$. Furthermore, there exists an isometry 
$$
\sigma: L \rightarrow \End(\Ecal)
$$
so that $\sigma(\nu)=\rho.$
\end{proof}

We also want to point out the following classical compactness result

\begin{thm}
The Uhlenbeck compactness (see Theorem \ref{NilpotentCompactness}) can be applied to a sequence of solutions to the Vafa-Witten equation over a compact K\"ahler manifold $(X,\omega)$ with $\deg K_X=0$. 
\end{thm}

\begin{proof}
    By Proposition \ref{deg=0}, all the solutions are Hermitian-Yang-Mills connections and the conclusion follows from one of the main results in \cite[Theorem $4.3.3$]{Tian:00}.
\end{proof}

\section{Useful estimates from the spectral cover}\label{Estimates}
We will generalize Mochizuki and Simpson's estimates (\cite{Mochizuki:16} \cite{Simpson:90}) for Higgs bundles over Riemann surfaces to the general case of K\"ahler manifolds. It is done by closely following Mochizuki's arguments (\cite[Section 2]{Mochizuki:16}) with adaptions to higher dimensions. See also \cite{He20} for related discussions.

\subsection{$C^0$ estimates on the Higgs fields from the control of spectral cover}
We work near $0\in B_R\subset \CBbb^n$ where $B_R$ is the ball of radius $R$ centered at the origin endowed with the standard flat metric \footnote{We assume this for simplicity of discussion and the estimates can be easily adapted for general K\"ahler metrics.}. Let $(A, \phi)$ be a solution to the Vafa-Witten equation defined near a neighborhood of $\overline{B_R}$. Write
$$
\phi=M dz_1 \wedge \cdots dz_n.
$$
Unless specified, all the constants below will only depend on the controlled geometry of the base but not on the pair $(A,\phi)$.

\begin{prop}\label{Prop4.1}
Suppose all the eigenvalues of $M$ are bounded by $S$ over $B_R$. Then 
$$
\|M\|_{C^0(B_{\frac{R}{2}})} \leq C_0(S+1)
$$
for some constant $C_0$.
\end{prop}

\begin{proof}
As the same computation as the proof for Proposition \ref{WeinzenbockFormula}, 
$$
\Delta_{\dbar} \log |M|^2\leq  \frac{([\sqrt{-1} \Lambda_{\omega} F_A,M], M)}{|M|^2}=-\frac{|[M,M^*]|^2}{|M|^2}.
$$
By Lemma \ref{MatrixNorm}, 
$$
|[M,M^*]| \geq C(|M|^2-\sum |\lambda_i|^2)\geq 0
$$
for some dimensional constant $C>0$ where $\{\lambda_i\}_i$ denote the list of eigenvalues of $M$ counted with multiplicities. Thus
$$
\Delta_{\dbar} \log |M|^2 \leq -C^2 \frac{(|M|^2-\sum |\lambda_i|^2)^2}{|M|^2}.
$$
Assume $|M|^2 > 2 \sum |\lambda_i|^2$ at a point $x$, then 
$$
\Delta_{\dbar} \log |M|^2 \leq - \frac{C^2}{4} |M|^2.
$$
Now we consider the auxiliary function $\log \frac{C'}{(R^2-|z|^2)^2}$ for some constant $C'$ to be determined later. A direct computation shows 
$$
\Delta_{\dbar} \log(\frac{C'}{(R^2-|z|^2)^2})= -\frac{2n R^2}{(R^2-|z|^2)^2}+\frac{2(n-1)|z|^2}{(R^2-|z|^2)^2} \geq -\frac{2nR^2}{C'}\frac{C'}{(R^2-|z|^2)^2}.
$$
In particular, we have at the points where $|M|^2 \geq 2 \sum_i |\lambda_i|^2$
$$
\Delta_{\dbar} [\log|M|^2- \log(\frac{C'}{(R^2-|z|^2)^2})] \leq -\frac{C^2}{4}[|M|^2-\frac{C'}{(R^2-|z|^2)^2}]
$$
where $C'$ is chosen so that $\frac{2nR^2}{C'}\leq \frac{C^2}{4}$. Choose further $C' \geq 2R^4 r S^2$. Let 
$$
U=\{z\in B_R: |M|^2 > \frac{C'}{(R^2-|z|^2)^2}\}
$$
which is a precompact open set in $B_R$. Suppose it is not empty. Then for any $z\in U$, we have 
$$
|M|^2> \frac{C'}{R^4} \geq 2r S^2>2 \sum_i |\lambda_i|^2
$$
which implies the following over the open set $U$
$$
\Delta_{\dbar} [\log|M|^2- \log(\frac{C'}{(R^2-|z|^2)^2})] < 0.
$$
Since $U$ is relatively compact in $B_R$, we obtain
$$
[\log|M|^2- \log(\frac{C'}{(R^2-|z|^2)^2})]|_{\partial U}=0
$$
and by the Maximum principle, it holds over $U$ that
$$
\log|M|^2- \log(\frac{C'}{(R^2-|z|^2)^2}) \leq 0
$$
which is a contradiction. Thus $U$ is empty. In particular,
$$
\|M\|_{C^0(B_{\frac{R}{2}})} \leq \frac{2\sqrt{C'}}{R^2}
$$
for any $C'$ so that $C' \geq \frac{8nR^2}{C^2}$ and $C'\geq 2R^4 r S^2$. The conclusion follows.
\end{proof}
In particular, this gives
\begin{cor}\label{Cor3.5}
$|\Lambda_{\omega} F_A| \leq C(S^2+1+\mu)$ for some constant $C$ where $S$ depends on only on the upper bounds of the eigenvalues for $M$.
\end{cor}

\subsection{Estimates from the splitting associated to the spectral data}
Let $(A,\phi)$ be a solution to the Vafa-Witten equation as above. Below we assume
$$
(\Ecal, \phi)=\oplus_{\lambda\in \Lambda} \Ecal_\lambda
$$
over a neighborhood of $\overline{B_R}$ 
so that 
\begin{itemize}
\item $M=\oplus_\lambda M_\lambda$ where $\lambda$ is a holomorphic function and 
$$M_\lambda: \Ecal_\lambda \rightarrow \Ecal_\lambda$$ 
has $\lambda$ as its eigenvalue with algebraic multiplicity equal to $\rank(\Ecal_\lambda)$;

\item $\lambda \neq \lambda'$ for any $\lambda, \lambda' \in \Lambda$;
\end{itemize}
Denote
\begin{itemize}
\item $S:=\max_{\lambda}\sup_{B_R}|\lambda|;$
\item $\delta=\inf_{\lambda \neq \lambda'} \inf_{B_R} |\lambda-\lambda'|$
\end{itemize}
Then we assume
\begin{center}
    \textbf{Assumption}: $\delta\geq 1 \text{ and } S \leq C \delta$ 
\end{center}
for some fixed constant $C$. For each $\lambda$, denote the \emph{holomorphic} projection to $\Ecal_\lambda$ through the splitting as
$$
p_\lambda: \Ecal \rightarrow \Ecal_{\lambda}
$$
and the \emph{orthogonal} projection to $\Ecal_\lambda$ induced by the metric $H$ as 
$$
\pi_\lambda: \Ecal \rightarrow \Ecal_\lambda.
$$
We first note 
\begin{lem}\label{C0bound}
There exists a constant $C_1>0$ so that 
$$
|p_\lambda|<C_1.
$$  
In particular, $|p_\lambda-\pi_\lambda|<C_1$.
\end{lem}

\begin{proof}
This follows exactly the same as \cite[Lemma 2.17]{Mochizuki:16} which is essentially a computation at an arbitrary point $z$.  For fixed $\lambda$, let $\gamma$ be a circle centered around $\lambda$ with radius equal to $\frac{\delta}{100}$. Then $\id_\Ecal-M$ is invertible along $\gamma$ since $M$ has no eigenvalues along $\gamma$. By Cauchy's integral formula, we obtain 
$$
p_\lambda=\frac{1}{2\pi \sqrt{-1}} \int_{\gamma} (\zeta\id_{\Ecal}-M)^{-1} d\zeta. 
$$
by using the simple fact that 
$$
(\zeta\id_{\Ecal}-M)^{-1}=\oplus_{\lambda'} (\zeta \id_{\Ecal_{\lambda'}}-M|_{\Ecal_{\lambda'}})^{-1}.
$$
Now the conclusion follows from
$$
|(\zeta\id_{\Ecal}-M)^{-1}| \leq C' (2 \delta)^{-r}. (\delta+1)^{r-1} \leq C'
$$
along the circle $\gamma_{\alpha}$ for some suitable choice of constant $C'$. 
\end{proof}

We need the following 

\begin{lem}\label{Lemma3.5}
For any holomorphic section $N$ of $\Hom(\Ecal, \Ecal)$ with $[M,N]=0$, the following holds
$$
\Delta_{\dbar} \log |N|^2 \leq -\frac{|[M^*,N|}{|N|^2}.
$$
\end{lem}

\begin{proof}
The statement follows from a direct computation as 
$$
\begin{aligned}
\Delta_{\dbar} \log |N|^2 & \leq  \frac{([\sqrt{-1} \Lambda_{\omega} F_A,N], N)}{|N|^2}\\
&=-\frac{([[M,M^*],N],N)}{|N|^2}\\
&=\frac{([[M^*,N],M]+[[N,M],M^*], N)}{|N|^2}\\
&=\frac{([[M^*,N],M], N)}{|N|^2}\\
&=-\frac{|[M^*N]|^2}{|N|^2}.
\end{aligned}
$$
where the first inequality is the same as Proposition \ref{Prop4.1}.
\end{proof}

Using this equation, we can improve the estimate on $|p_\lambda-\pi_\lambda|$.
\begin{prop}\label{Prop3.6}
For some constant $C_2>0$ and $\epsilon_0>0$, the following holds 
$$|p_\lambda-\pi_\lambda| \leq C_2 e^{-\epsilon_0  \delta}$$ 
over $B_{\frac{R}{4}}$. 
\end{prop}

\begin{proof}
We first note 
$$
|\pi_\lambda|^2 +|p_\lambda-\pi_\lambda|^2= |p_\lambda|^2
$$
where $|\pi_\lambda|^2 =r_{\lambda}:=\rank \Ecal_\lambda$, thus
$$
1+\frac{|p_\lambda-\pi_\lambda|^2}{r_{\lambda}}=\frac{|p_\lambda|^2}{r_\lambda}.
$$
By Lemma \ref{C0bound}, since $p_\lambda-\pi_\lambda$ are bounded, this implies 
$$
(C')^{-1} \frac{|p_\lambda-\pi_\lambda|^2}{r_{\lambda}} \leq \log \frac{|p_\lambda|^2}{r_\lambda} \leq C' \frac{|p_\lambda-\pi_\lambda|^2}{r_{\lambda}}.
$$
It suffices to control $\log \frac{|p_\lambda|^2}{r_\lambda}.$ 
By Lemma \ref{Lemma3.5}, we obtain 
$$
\Delta_{\dbar} \log \frac{|p_\lambda|^2}{r_\lambda}\leq -\frac{|[M^*,p_\lambda]|^2}{|p_\lambda|^2} \leq -\epsilon' \delta^2   \log \frac{|p_\lambda|^2}{r_\lambda}.
$$
for some constant $\epsilon'$ where the second inequality follows from a linear algebra fact that (see \cite[Lemma 2.8]{Mochizuki:16})
$$|[M^*, p_\lambda]| \geq C'' \delta |p_\lambda-\pi_\lambda|$$
for some constant $C''=C''(\rank \Ecal, S)$ and the inequality above. Consider the comparison function $e^{\epsilon'' \delta |z|^2}$ which satisfies 
$$
\Delta_{\dbar} e^{\epsilon'' \delta |z|^2} \geq -\epsilon' \delta^2 e^{\epsilon'' \delta |z|^2} 
$$
over $B_{\frac{R}{2}}$ for suitable choice $\epsilon''$. Now we choose $C''$ so that 
$$
[C'' - \log \frac{|p_\lambda|^2}{r_\lambda}]|_{\partial B_{\frac{R}{2}}}>0,
$$
i.e.,
$$
C''>\min_{\partial B_{\frac{R}{2}}} \log\frac{|p_\lambda|^2}{r_\lambda}
$$
Consider 
$$
U=\{z\in B_{\frac{R}{2}}: C'' e^{-\epsilon'' \delta \frac{R^2}{4}} e^{\epsilon'' \delta |z|^2} < \log \frac{|p_\lambda|^2}{r_\lambda}\}
$$
which is precompact and open in $B_{\frac{R}{2}}$. Over $U$, we obtain 
$$
\Delta_{\dbar}(\log \frac{|p_\lambda|^2}{r_\lambda}-C'' e^{-\epsilon'' \delta \frac{R^2}{4}} e^{\epsilon'' \delta |z|^2})< 0.
$$
On the other hand, by the Maximum principle, it holds over $U$ that
$$
\log \frac{|p_\lambda|^2}{r_\lambda} \leq C'' e^{-\epsilon'' \delta \frac{R^2}{4}} e^{\epsilon'' \delta |z|^2} 
$$
which is a contradiction to the definition of $U$. Thus $U=\emptyset$. In particular, over $B_\frac{R}{2}$
$$
\log \frac{|p_\lambda|^2}{r_\lambda} \leq C'' e^{-\epsilon'' \delta \frac{R^2}{4}} e^{\epsilon'' \delta |z|^2} 
$$
Thus over $B_{\frac{R}{4}}$, we have 
$$
\log \frac{|p_\lambda|^2}{r_\lambda} \leq C'' e^{-\epsilon'' \delta \frac{R^2}{4}} e^{\epsilon'' \delta |z|^2} \leq C'' e^{-\epsilon'' \delta \frac{3R^2}{16}}.
$$
The conclusion follows by taking $C_2=C''$ and $\epsilon_0=\epsilon'' \frac{3R^2}{16}$.
\end{proof}

\begin{prop}\label{Prop3.8}
There exists some constant $C_4$ and $\epsilon_2 >0$ so that over $B_{\frac{R}{8}}$
$$
\int_{B_{\frac{R}{8}}} |\partial_A p_\lambda|^2=\int_{B_{\frac{R}{8}}}|\dbar_A p_{\lambda}^*|^2 \leq C_4 e^{-\epsilon_2 \delta}.
$$
\end{prop}

\begin{proof}
A direct computation using $\dbar_{A} p_\lambda=0$ shows 
$$
\begin{aligned}
\Delta_{\dbar}|p_\lambda-p_\lambda^*|^2&=([\sqrt{-1} \Lambda_{\omega} F_A, p_\lambda], p_\lambda-p_\lambda^*)-|\partial_A p_\lambda|^2\\
\end{aligned}
$$
Take $\chi$ to be any cut-off function so that 
$$
\chi|_{B_{\frac{R}{8}}}=1, \chi|_{B^c_{\frac{R}{4}}}=0, |\nabla \chi| \leq \frac{C}{R}, |\nabla^2 \chi|\leq \frac{C}{R^2}.
$$
Multiplying the above by $\chi$ and doing integration, we obtain 
$$
\begin{aligned}
&\int_{B_{\frac{R}{8}}} |\partial_A p_\lambda|^2\\
\leq & |\int_{B_{\frac{R}{4}}} \chi \Delta_{\dbar}|p_\lambda-p_\lambda^*|^2| + |\int_{B_{\frac{R}{4}}} \chi([\sqrt{-1} \Lambda_{\omega} F_A, p_\lambda], p_\lambda-p_\lambda^*)|\\
\leq &\int_{B_{\frac{R}{4}}} |p_\lambda-p_\lambda^*|^2  |\Delta_{\dbar}\chi| +\int_{B_{\frac{R}{4}}} |[\sqrt{-1} \Lambda_{\omega} F_A, p_\lambda]| |p_\lambda-p_\lambda^*|\\
\end{aligned}
$$
Now using $\pi_\lambda=\pi_\lambda^*$, we obtain by Proposition \ref{Prop3.6} that
$$
|p_\lambda-p_\lambda^*| \leq |p_\lambda-\pi_\lambda|+|\pi_\lambda^*-p_\lambda^*| \leq 2C_2 e^{-\epsilon_0 \delta}
$$
and we also know from Corollary \ref{Cor3.5} that 
$$
|\sqrt{-1} \Lambda_{\omega} F_A|\leq C(1+S^2+\mu). 
$$
In particular, for suitable choice of $\epsilon_3$, we have 
$$
\int_{B_{\frac{R}{8}}} |\partial_A p_\lambda|^2 \leq C_4 e^{-\epsilon_3 \delta}.
$$
\end{proof}

\begin{cor}\label{Cor3.12}
There exists some constant $C_5$ and $\epsilon_3 >0$ so that over $B_{\frac{R}{8}}$
$$
\int_{B_{\frac{R}{8}}} |\dbar_A \pi_\lambda|^2=\int_{B_{\frac{R}{8}}}|\partial_A \pi_{\lambda}|^2 \leq C_5 e^{-\epsilon_3 \delta}.
$$
\end{cor}

\begin{proof}
By Proposition \ref{Prop3.8}, it suffices to show 
$$
|\dbar_A \pi_\lambda|\leq |\dbar_A p_{\lambda}^*|.
$$
Indeed, we obtain
$$
\dbar_A \pi_\lambda=\dbar_A(\pi_\lambda-p_\lambda).
$$
Combined with the orthogonal decomposition
$$
-\dbar_A(p_\lambda^*)=\dbar_A(p_\lambda-p_\lambda^*)=\dbar_A(p_\lambda-\pi_\lambda)+\dbar_A(\pi_\lambda-p_\lambda^*),
$$
this implies $|\dbar_A \pi_\lambda|\leq |\dbar_A p_{\lambda}^*|$.
\end{proof}

\subsection{Applied to a sequence when the spectral covers become unbounded}
Let $\{(A_i, \phi_i)\}_{i\in \NBbb}$ be a sequence of solutions to the Vafa-Witten equations over $B_{2R}$ satisfying 
\begin{itemize}
    \item $b_k(\rho_i) \rightarrow b_k^\infty$ where $\rho_i = r_i^{-1} \phi_i$ and $r_i \rightarrow \infty$ as $i\rightarrow 0$;
    \item $\Delta(b^\infty)=\emptyset$ where $b^\infty=(b_1^\infty, \cdots b^\infty_{\rank E})$.
\end{itemize}
In particular, for $i$ large, $\phi_i$ are regular and semi-simple over $\overline{B_R}$ and we can decompose 
$$
(E, \dbar_{A_i})=\oplus_k \Ecal^i_{\lambda_k}
$$
where $\Ecal_{\lambda_k}$ has rank one and by writing $\phi_i=M^i dz_1 \cdots dz_n$ locally $M^i$ acts on $\Ecal_{\lambda_k}$ by scaling with $\lambda_k$. Denote $\pi^i_k: (E, \dbar_{A_i}) \rightarrow \Ecal^i_{\lambda_k}$ and $p^i_k: (E, \dbar_{A_i}) \rightarrow \Ecal^i_{\lambda_k}$ as the holomorphic and orthogonal projections respectively.
\begin{prop}\label{Prop3.13}
The following holds 
\begin{itemize}
    \item $\lim_i \sup_{B_{\frac{R}{8}}} |p^i_k-\pi^i_k| \rightarrow 0$;
    \item $\lim_i \int_{B_{\frac{R}{8}}} |\nabla_{A_i} \pi^i_k|^2 \rightarrow 0$;
    \item $\lim_i \int_{B_{\frac{R}{8}}} |\nabla_{A_i} p^i_k|^2 \rightarrow 0$.
\end{itemize}
\end{prop}

\begin{proof}
    This directly follows from applying the estimates above to $\pi^i_k$ and $p^i_k$. Denote 
    $$
    \delta^i=\max_{k\neq k'} \sup_ {\overline{B_R}}|\lambda^i_k - \lambda^i_{k'}|
    $$
    and 
    $$
    S^i=\max_{k\neq k'} \sup_{\overline{B_R}} |\lambda^i_k|.
    $$
    Since $b_k(r_i^{-1} \phi_i^k) \rightarrow b_k^\infty$ for each $k$ and $\Delta(b^\infty)=\emptyset$ over $B_{2R}$, we obtain the spectrum of $r_i^{-1}\phi_i$ also converges and we denote  $r_i^{-1}\lambda_k^i \rightarrow \lambda_k^\infty$. Then for any $k\neq k'$   
    $$
    \lambda_k^\infty \neq \lambda_{k'}^\infty
    $$
    thus one can easily see that 
    $$
    \delta^i \rightarrow \infty 
    $$
    as $i\rightarrow \infty$ and 
    $$
    S^i \leq C\delta^i
    $$
    for some $C$ independent of $i$. Now the conclusion follows from Proposition \ref{Prop3.6}, Proposition \ref{Prop3.8} and Corollary \ref{Cor3.12} applied to the sequence.
\end{proof}

\section{Uhlenbeck compactness for solutions with uniformly bounded spectral covers} \label{Nilpotentcase} 
In the following, we will always consider a sequence of solutions $\{(A_i, \phi_i)\}_{i\in \NBbb}$ to the Vafa-Witten equation on a fixed unitary bundle $(E, H)$ over a compact K\"ahler manifold. We first note the following 
\begin{prop}\label{prop4.1}
Assume $\phi_i$ are all nilpotent. There exists a dimensional constant $C$ so that for any nilpotent solution $(A,\phi)$ to the Vafa-Witten equation
$$
\|\phi\|_{C^0}\leq C.
$$ 
\end{prop}

\begin{proof}
By Proposition \ref{WeinzenbockFormula}, we obtain 
$$
\Delta_{\dbar} |\phi|^2 \leq  -|[\phi;\phi]|^2+2\pi\deg(K_X) |\phi|^2.
$$
Since $\phi$ is nilpotent, i.e., $b_k(\phi)\equiv 0$ for any $k$, by Lemma \ref{MatrixNorm}, we obtain 
$$
|\phi|^4 \lesssim |[\phi;\phi]|^2
$$
thus combined with the above, this implies 
$$
\Delta_{\dbar} |\phi|^2 +|\phi|^4 \lesssim 2\pi\deg(K_X) |\phi|^2.
$$
Suppose $|\phi|$ achieves its maximum at $x_m$, then 
$$
|\phi|^4(x_m) \lesssim 2\pi\deg(K_X) |\phi|^2(x_m).
$$
which implies 
$$
|\phi|(x_m) \lesssim \sqrt{2\pi \deg K_X}.
$$
\end{proof}
Actually, this also follows from the $C^0$ estimate we derived in Proposition \ref{Prop4.1} which gives the following 
\begin{lem}
    Assume the spectral covers of $X_{b(\phi_i)}$ are uniformly bounded, then the Higgs fields $\phi_i$ are uniformly bounded. 
\end{lem}

In particular, this gives
\begin{cor}
 Assume the spectral covers of $X_{b(\phi_i)}$ are uniformly bounded, then $\|F_{A_i}\|_{L^2} \leq C$ and $\|\Lambda_\omega F_{A_i}\|_{C^0} \leq C$.
\end{cor}

Given the estimates above, by Uhlenbeck's compactness results (\cite{UhlenbeckPreprint} \cite[Theorem $5.2$]{UhlenbeckYau:86}), we obtain 
\begin{prop}
Given any sequence of solutions $(A_i, \phi_i)$ to the Vafa-Witten equation with uniformly bounded spectral covers, passing to a subsequence, up to gauge transforms, $(A_i, \phi_i)$ converges locally smoothly to $(A_\infty, \phi_\infty)$ over $X\setminus Z$ where 
$$
Z=\{x\in X: \lim_{r\rightarrow 0} \liminf_{i}\int_{B_x(r)} |F_{A_i}|^2 \dvol \geq \epsilon_0  \}
$$
is a closed subset of $X$ of Hausdorff codimension at least four. Here $\epsilon_0$ denotes the regularity constant\footnote{If the quantity above is smaller than $\epsilon_0$, then by fixing gauge, one can get higher order estimates for the connections. }.
\end{prop}

\begin{rmk}
In \cite{UhlenbeckPreprint} and \cite{UhlenbeckYau:86}, the convergence is in $L^p_{1, loc}$ for general admissible connections for any fixed $p>1$. Since in our case the Vafa-Witten equations are elliptic after fixing gauge, the convergence can be improved to be $C^\infty_{loc}$ by standard bootstrapping arguments similar as the Yang-Mills case (\cite{Uhlenbeck:82a} \cite{Nakajima:88}). 
\end{rmk}

To finish the proof of Corollary \ref{NilpotentCompactness}, we need to show that the bubbing set is complex analytic which follows from similar argument as \cite{Tian:00} \cite{HongTian:04} and we include a sketched explanation.
\begin{prop}
$Z$ is a codimension at least two subvariety of $X$. Furthermore, $A_\infty$ defines a unique reflexive sheaf $\Ecal_\infty$ over $X$ and $\phi_\infty$ extends to be a global section of $Hom(\Ecal_\infty, \Ecal_\infty) \otimes K_X$ over $X$;  $Z$ is a codimension at least two subvariety that admits a decomposition 
$$
Z=\cup_k Z_k \cup \Sing(\Ecal_\infty)
$$
where $Z_k$ denotes the irreducible pure codimension two components of $Z
$. As a sequence of currents 
$$
\tr(F_{A_i} \wedge F_{A_i}) \rightarrow \tr(F_{A_\infty} \wedge F_{A_\infty})+ 8\pi^2 \sum m_k Z_k
$$
where $m_k \in \ZBbb_+$
\end{prop}

\begin{proof}

The argument follows exactly as the Hermitian-Yang-Mills case by Tian (\cite{Tian:00}) combined with the results by Uhlenbeck \cite{UhlenbeckPreprint} and Bando-Siu (\cite{BandoSiu:94}). More precisely, the main results of \cite{UhlenbeckPreprint} states that the monotonicity formula and $\epsilon$ regularity holds for a sequence of integrable unitary connections with bounded Hermitian-Yang-Mills/Hermitian-Einstein tensor and bounded $L^2$ norm of the curvature. With these two key ingredients, the argument for Yang-Mills connections in \cite[Section $3$]{Tian:00} can be used to show $Z$ is $2n-4$ rectifiable, i.e., it has Hausdorff codimension at least $2n-4$ and locally can written as a countable union of images of continuously differentiable maps $\{f_i: \RBbb^{2n-4} \rightarrow X\}_i$ away from $(2n-4)$-Hausdorff  measure zero set. Now except the extension part, the conclusion follows as \cite[Theorem $4.3.3.$]{Tian:00} and \cite[Theorem 11]{HongTian:04}. The fact that $A_\infty$ defines a reflexive sheaf follows from \cite[Theorem $2$]{BandoSiu:94} and the extension of $A_\infty$ follows from \cite[Proposition 1]{BandoSiu:94}. The extension of $\phi_\infty$ follows from the Hartog's property of sections of reflexive sheaves (\cite[page 160]{Kobayashi:87}).
\end{proof}

\section{Compactness for rank two solutions with trace free Higgs fields}\label{Section 6}
\subsection{$\ZBbb_2$ Holomorphic $n$-forms} We first define the notion of $\ZBbb_2$ holomorphic $n$-forms by generalizing Taubes' notion of $\ZBbb_2$ harmonic $2$-forms.
\begin{defi}\label{SpinorZ2}
A $\ZBbb_2$ holomorphic $n$-form is defined as a triple $(L, \rho, Z)$ satisfying
\begin{itemize}     
\item $Z$ is a Hausdorff codimension at least two closed subset of $X$;

\item $L$ is the natural complex unitary flat line bundle associated to a principal $\ZBbb_2$ bundle over $X\setminus Z$;
 
\item $\rho\in H^0(X\setminus Z, L\otimes  K_X)$ and $|\rho|$ extends to a H\"older continuous function over $X$;
\item $Z=|\rho|^{-1}(0)$. 
\end{itemize}
\end{defi}

Actually the definition tells more  
\begin{lem}
Given a nonzero $\ZBbb_2$ holomorphic $n$-form $(L, \rho, Z)$, the following holds 
\begin{itemize}
\item $\rho^2$ can be extended to be a holomorphic section of $K^2_X$ over $X$;
\item $Z$ is a subvariety of pure codimension one;
\item the principal $\ZBbb_2$ bundle associated to $L$ arises as the cyclic covering by taking the square root of $\rho^2$.
\end{itemize}
\end{lem}

\begin{proof}
Since $L$ is defined by a principal $\ZBbb_2$ bundle, $L^{\otimes 2} = \Ocal_X$ as a holomorphic line bundle. Thus $\rho^{\otimes 2}$ gives a holomorphic section of $(L\otimes K_X)^{\otimes 2}$ over $X\setminus Z$. By the extension property of analytic functions (\cite{Besicovitch1931}), we obtain $\rho^{\otimes 2}$ can be extended to be a global section of $K_X^{\otimes 2}$ over $X$. In particular, we obtain $Z=\rho^{-1}(0)$ is either empty or a hypersurface. The last statement follows from definition. 
\end{proof}
In particular, this gives the following nonexistence results
\begin{cor}
Suppose $\deg K_X < 0$, then there does not exist any nontrivial $\ZBbb_2$ holomorphic $n$-forms
\end{cor}
\begin{proof}
    Otherwise, since $\rho^2$ gives a nontrivial section of $K_X^{\otimes 2}$, $\deg(K_X) \geq 0$ which is a contradiction.
\end{proof}

Fix $(A, \phi)$ to be a rank two solution of the Vafa-Witten equation with 
$$
\tr(\phi)=0 \text{ and } \phi \neq 0.
$$
Denote
$$
Z=(b_2(\phi))^{-1}(0).
$$
We have the following 
\begin{lem}\label{Lemma6.4}
There exists natural $\ZBbb_2$ holomorphic $n$-forms $(L_{b(\phi)}, \nu, Z)$ associated to the spectral cover $X_{b(\phi)}$. 
\end{lem}

\begin{proof}
Indeed, for any $x\in X\setminus Z$, we obtain $b_1(\phi)=0$ and $b_2(\phi)=\det \phi \neq 0$ near $x$. Thus locally near $x$, the spectral cover is a trivial $2$-to-$1$ cover and we can decompose 
$$
\Ecal=\Lcal_{\nu^x} \oplus \Lcal_{-\nu^x}
$$
where $\nu^x$ is a nonzero holomorphic $(n,0)$ form near $x$ and $\phi$ acts on $L_{\nu}$ by tensoring with $\nu^x$. Now we can cover $X\setminus Z$ with such open sets and get data $\cup_x (U_x, \nu^x)$ which can be naturally viewed as a holomorphic section of $L\otimes K_X$ over $X\setminus Z$. We denote the section as $\nu$. Then we obtain 
$$
|\nu|^2=|b_2(\phi)|
$$
thus $|\nu|=\sqrt{|b_2(\phi)|}$ is a H\"older continuous function across $Z$.
\end{proof}

\begin{rmk}
In general, given any $b_2\in H^0(X, K_X^{\otimes 2})$, one can take the cyclic covering defined by the square root of $b_2$ in $K_X$. The $\ZBbb_2$ holomorphic $n$-form associated to $X_b$ is then union of local choices of $\sqrt{-b_2}$ away from the discriminant locus $\Delta_b$.
\end{rmk}

\subsection{Sequential limits}
In the following, we will consider a sequence of rank two solutions $(A_i, \phi_i)$ to the Vafa-Witten equation with 
$$
\tr(\phi_i)=0
$$
and 
$$
\lim_i \|\phi_i\|_{L^2} = \infty. 
$$
Denote
$$
\rho_i=\frac{\phi_i}{\|\phi_i\|}
$$
and $r_i=\|\phi_i\|$, then $(A_i, \rho_i)$ satisfies 
\begin{itemize}
\item $F^{0,2}_{A_i}=0$;
\item $\sqrt{-1} \Lambda_{\omega} F_{A_i}+r_i^2[\rho_i;\rho_i]=c\id$;
\item $\tr(\rho_i) =0$;
\item $\lim r_i =\infty$
\end{itemize}
We start with the following simple observation from Proposition \ref{WeinzenbockFormula}.
\begin{lem}
$\|[\rho_i;\rho_i]\|_{L^2} \leq \frac{1}{r_i}$. 
\end{lem}
By standard elliptic theory, passing to a subsequence, we can always assume 
$$
b_2(\rho_i) \rightarrow b_2^\infty.
$$
We need the following simple observation
\begin{prop}\label{Cor5.2}
$b_2^\infty \neq 0$.
\end{prop}

\begin{proof}
Indeed, we obtain 
$$
|\rho_i|^2 \sim |[\rho_i;\rho_i]|+|b_1(\rho_i)|^2 + |b_2(\rho_i)|= |[\rho_i;\rho_i]| + |b_2(\rho_i)|
$$
thus 
$$
1\sim \int_X |[\rho_i;\rho_i]|+ \int_{X} |b_2(\rho_i)|
$$
Since $b_2(\rho_i) \rightarrow b_2^\infty$ smoothly over $X$, by taking the limit and using that
$$
\int_X|[\rho_i;\rho_i]| \rightarrow 0,
$$ 
we obtain 
$$
\int_{X} |b_2^\infty| \neq 0.
$$
The conclusion follows.
\end{proof}

Given this, we can take the spectral cover $X_{b^\infty}$ (resp. $X_{b(\rho_i)}$) defined by $b^\infty=(b_1^\infty, b_2^\infty)$ (resp. $b(\phi_i)=(b_1(\rho_i), b_2(\phi_i))$). 
\begin{prop}
By passing to a subsequence, $X_{b(\rho_i)}$ converges naturally to $X_{b^\infty}$. 
\end{prop}

\begin{proof}
This follows from the simple fact that $b(\rho_i) \rightarrow b^\infty$ smoothly. Thus $\sqrt{b(\rho_i)} \rightarrow \sqrt{b^\infty}$ smoothly.
\end{proof}

Now we are ready to prove the main result in the rank two case.

\begin{proof}[Proof of Theorem \ref{Rank=2}]
Let $Y\subset X\setminus Z$ be any compact subset. Fix any $x\in Y$, i.e., $b_2^\infty(x) \neq 0$. Since $b_2(\rho_i) \rightarrow b_2^\infty$ smoothly, we obtain for $i$ large, $b_2(\rho_i)(x) \neq 0$, thus locally near $x$, the spectral cover is a trivial $2$-to-$1$ cover and we can decompose 
$$
\Ecal_i=\Lcal_{\nu^x_i} \oplus \Lcal_{-\nu^x_i}
$$
where $\nu^x_i$ is a nonzero holomorphic $(n,0)$ form near $x$ and $\rho_i$ acts on $L_{\nu}$ as tensoring by $\nu^x_i$. Furthermore $\nu^x_i$ converges to a square root $\nu^x_\infty$ of $b_2^\infty(x)$. By taking a cover of $X\setminus Z$ and making a choice of square root, we get data $\cup_x (U_x, \nu^x_\infty)$ away from $Z$ which defines a $\ZBbb_2$ holomorphic $n$-form $(L, \nu_\infty, \Delta(b_\infty))$ associated to the spectral cover $X_{b_\infty}$. Now the construction of $\sigma_i$ is simply by locally defining over $Y$
$$
\sigma_i^x = \pi_{i+}^x-\pi_{i-}^{x}
$$
over $U_x$ where $x\in Y$ and $\pi_{i\pm}^x$ denotes the orthogonal projections to $\Lcal_{\pm \nu_i^x}$ respectively. We prove $\sigma_i$ satisfies the properties required over $Y$.  For this, denote 
$p_{i\pm}: \Ecal \rightarrow \Lcal_{\pm\nu_i}$
to be the local holomorphic projections. Write 
$$
\sigma_i \nu_\infty-\rho_i=\nu_\infty \pi_{i+}^x-\nu_\infty \pi_{i-}^x-(\nu_i p^x_{i+}-\nu_i p^x_{i-}).
$$
By Proposition \ref{Prop3.13}, we obtain locally on a smaller open set
$$
|p_{i\pm}^x| \leq C, |\pi_{i\pm}^x| \leq C,
$$
and 
$$
|\pi^x_{\pm}-p^x_{\pm}| \xrightarrow{C^0_{loc}(X\setminus Z)} 0, |\nabla_{A_i} \pi^x_{i\pm}| \xrightarrow{L^2_{loc}(X\setminus Z)} 0, |\nabla_{A_i} p_i^{\pm}| \xrightarrow{L^2_{loc}(X\setminus Z)} 0
$$
as $i\rightarrow \infty$ near $x$. Combined with $\nu_i \rightarrow \nu_\infty$, we get the desired properties over $Y$. Since we can take $Y$ to be a deformation restraction of $X\setminus Z$, $\sigma_i$ can be extended to be over $X\setminus Z$.  To get the statement about the convergence over $X\setminus Z$, this follows from a standard diagonal argument by exhausting $X\setminus Z$ with countably many precompact open subsets.
\end{proof}

Assume $\dim_{\CBbb} X =2$. This recovers Taubes' results  in the K\"ahler surface case (\cite{Taubes17}). Here by assuming $\dim_{\CBbb}X=2$, the pair $(A, a=\phi-\phi^*)$ satisfies the real version of the Vafa-Witten equation over a Riemannian four manifold
\begin{itemize}
\item $d_Aa=0$;
\item $F_A^+=\frac{1}{2}[a;a]$.
\end{itemize}
In particular, this gives an algebraic characterization of the limiting data in Taubes' analytic results over K\"ahler surfaces.

\begin{cor}\label{TaubesZ2}
Assume $(X, \omega)$ is a K\"ahler surface. The $\ZBbb_2$ harmonic $2$-form obtained by Taubes in \cite{Taubes17} is determined by the $\ZBbb_2$ holomorphic $2$-form associated to the spectral cover $X_{b^\infty}$.
\end{cor}

\begin{proof}
Let $\rho$ be the limiting holomorphic $\ZBbb_2$ $2$-form obtained in our result. Take $\nu=Re(\rho)$ which is then closed and anti-self-dual (\cite[Page 47]{DonaldsonKronheimer:90}). In particular, $\nu$ is harmonic. This then recovers the notion of $\ZBbb_2$ harmonic  2-forms in Taubes' results for Vafa-Witten equations over K\"ahler surfaces with structure group $\SU(2)$. 
\end{proof}

\begin{rmk}
    It is very straightforward to generalize the discussion in this section to the case when the spectral cover is cyclic, i.e., $b_1=\cdots b_{r-1}=0$.  
\end{rmk}

\section{Compactness in general}
In this section, we will explain the role of the spectral cover in general when studying the limit of the renormalized Higgs fields. The rank two case with trace free Higgs fields obtained in Section \ref{Section 6} can also be seen as a variant of this by taking into consideration the special symmetry of the spectral cover. 

Below we assume $(A_i, \phi_i)$ is a sequence of solutions to the Vafa-Witten equation on a fixed unitary bundle $(E,H)$ over a compact K\"ahler manifold $(X, \omega)$ satisfying
$$
\lim_i \|\phi_i\| = \infty
$$
and we denote 
$$
\rho_i = \frac{\phi_i}{\|\phi_i\|}.
$$

\begin{lem}
By passing to a subsequence, $b_k(\rho_i)$ converges to $b_k^\infty\in H^0(X, K_X^{\otimes k})$ smoothly over $X$ for any $k$. Furthermore, $b_k^\infty \neq 0$ for some $k$. 
\end{lem}

\begin{proof}

The convergence follows from elliptic theory together with the simple fact that 
$$
|b_k(\rho_i)| \leq |\rho_i|^k
$$
which implies $\int_{X} |b_k(\rho_i)|^{\frac{2}{k}}\leq 1$. For the nontriviality of the limit, by Lemma \ref{MatrixNorm}, we obtain 
$$
|\rho_i|^2 \sim |[\rho_i;\rho_i]|+|b_1(\rho_i)|^2 + \cdots |b_n(\rho_i)|^{\frac{2}{n}}.
$$
Since by assumption 
$$
\|[\rho_i;\rho_i]\|_{L^2} \leq \frac{1}{r_i}
$$
which limits to zero while $\|\rho_i\|_{L^2}=1$, we obtain from the above that 
$$
\lim_i \int |[\rho_i;\rho_i]|+|b_1(\rho_i)|^2 + \cdots |b_n(\rho_i)|^{\frac{2}{n}} \sim 1
$$
while implies $b_k^\infty \neq 0$ for some $k\geq 1$.
\end{proof}

Now we can take the spectral cover 
$X_{b^\infty}$
where $b^\infty=(b_1^\infty, \cdots b^\infty_n)$. For the discussion below, we assume 
$$
\Delta_{b^\infty}\neq X
$$
which is equivalent to that the spectral cover $X_{b^\infty}$ is generically regular semi-simple. 

\begin{lem}\label{GenericallyRegularSemisimple}
For $i$ large, the spectral cover $X_{b(\rho_i)}$ is generically regular semi-simple. 
\end{lem}

\begin{proof}
This follows from that $b_k(\rho_i)\rightarrow b_k^\infty$. Thus the spectral cover converges to $X_{b(\rho_\infty)}$ at a generic point in a natural way. In particular, $\rho_i$ is generally regular semi-simple for $i$ large.
\end{proof}

Recall from the introduction, over the total space of the canonical bundle $K_X$, there exists a tautological line bundle $\Kcal\cong \pi^* K_X$ which has a global tautological section $\tau$ over $K_X$. Given any spectral cover $\pi: X_b \rightarrow X$ associated to $b\in \oplus_{i=1}^r H^0(X, K_X^{\otimes i})$, we denote 
$$
\Kcal^b=\Kcal|_{X_b}.
$$
Then $\pi_* \Kcal^b$ is a locally free sheaf of rank equal to $r$ away from the discriminant locus $\Delta_b$ and it has a tautological section 
$$
\tau_b=\pi_*(\tau|_{X_b}).
$$
Furthermore, it has a natural Hermitian metric induced by $K_X$ away from $\Delta_{b}$. 

Now we are ready to finish the proof of Theorem \ref{Main} which is essentially similar to the rank two case combined with the general analytic estimates in Proposition \ref{Prop3.13}. 

\begin{proof}[Proof of Theorem \ref{Main}]
Let $Y\subset X\setminus \Delta_{b_\infty}$ be any compact subset. Fix an arbitrary point $x\in Y$. By Lemma \ref{GenericallyRegularSemisimple}, for $i$ large, near $x$, we can decompose 
$$
(E, \dbar_{A_i})=\oplus_k \Lcal_{\lambda^k_i}
$$
where $\lambda^k_i$ is a holomorphic $(n,0)$ form near $x$ and $\rho_i$ acts on $\Lcal_{\lambda^k_i}$ as tensoring by $\lambda^k_i$. We also know $\lambda^k_i$ converge to some $\lambda^k_\infty \in X_{b^\infty}$. Define over $Y$
$$
\sigma_i: \pi_* \Kcal^{b_\infty} \rightarrow End(E) \otimes K_X
$$ 
by requiring
$$
\sigma_i(\lambda^k_\infty)=\pi_i^k \otimes \lambda^k_i
$$
over $U_x$ where $\pi_{i}^k$ denotes the orthogonal projections to $\Lcal_{\lambda_i^k}$. This fines $\sigma_i$ uniquely. Then for $i$ large, $\sigma^i$ can be glued together as a global section which we still denote as $\sigma_i: \pi_* \Kcal^{b_\infty} \rightarrow End(E) \otimes K_X$. We now prove that $\sigma_i$ satisfies the properties required.  For this, denote  the local holomorphic projections to $\Lcal_{\lambda^k_i}$ as $p_{i}^k: \Ecal \rightarrow \Lcal_{\lambda^k_i}$. we obtain then locally near $x$ as above
$$
\sigma_i \tau^{b_\infty}-\rho_i=\lambda_\infty^k \pi_{i}^k-\lambda_\infty^k \pi_{i}^k-(\lambda_i^k p^k_{i}-\lambda_i^k p^k_{i}).
$$
By Proposition \ref{Prop3.13}, we obtain locally on a smaller open set
$$
|p_{i}^k| \leq C, |\pi_{i}^k| \leq C,
$$
and 
$$
|\pi^k_{i}-p^k_{i}| \xrightarrow{C^0_{loc}(X\setminus \Delta_{b_\infty})} 0, |\nabla_{A_i} \pi^k_{i}| \xrightarrow{L^2_{loc}(X\setminus \Delta_{b_\infty})} 0, |\nabla_{A_i} p_i^{k}| \xrightarrow{L^2_{loc}(X\setminus \Delta_{b_\infty})} 0
$$
as $i\rightarrow \infty$ near $x$. Combined with $\lambda_i^k \rightarrow \lambda_\infty^k$, we get the desired properties for $\sigma_i$ over $Y\subset X\setminus \Delta_{b_\infty}$. Since we can take $Y$ to be a deformation restraction of $X\setminus Z$, $\sigma_i$ can be extended to be over $X\setminus Z$.  To get the statement about the convergence over $X\setminus Z$, this follows from a standard diagonal argument by exhausting $X\setminus Z$ with countably many precompact open subsets.
\end{proof}

\section{Examples}
We will analyze a few examples of moduli spaces of solutions to the Vafa-Witten equations by focusing on the rank two case. 
\subsection{Moduli space of $\SU(2)$ monopoles}
In this section, we study some properties of the moduli space $\Mcal^{mon}$ of monopoles mod gauge equivalence. The structure group of the bundle will be taken as $\SU(2)$.

\begin{lem}\label{Lem7.2}
Suppose $\Lcal$ is a holomorphic line bundle over $(X, \omega)$ with $c_1(\Lcal).[\omega]^{n-1}>0$. Then the Higgs bundle $(\Ecal= \Lcal \oplus \Lcal^{-1}, \phi)$ where 
$$
\phi=\begin{pmatrix}
0&0\\
\beta&0
\end{pmatrix}
$$
and $0\neq \beta \in H^0(X, K_X \otimes \Lcal^{-2})$ is a monopole.  
\end{lem}

\begin{proof}
By Proposition \ref{prop2.16}, $(\Ecal, \phi)$ is $\CBbb^*$ invariant. It remains to show that $(\Ecal, \phi)$ is stable. Suppose $\Lcal'$ is a rank one subsheaf of $\Ecal$ invariant under $\phi$. Then we observe the natural projection must map $\Lcal' \rightarrow \Lcal^{-1}$ nontrivially since $\phi$ maps $\Lcal$ to $K_X\otimes \Lcal^{-1}$ nontrivially. In particular, $\mu(\Lcal')<0$ and $\Ecal$ is stable. This finishes the proof.  
\end{proof}

Now we have 

\begin{prop}\label{Prop7.3}
The $\SU(2)$ monopoles are exactly of the form $(B, \beta)$ where  
\begin{itemize}
\item $B$ is a unitary connection on a line bundle $L$ with $F^{0,2}_B=0$;
\item $F_B \wedge \frac{\omega^{n-1}}{(n-1)!} = \beta \wedge \bar{\beta}$. 
\end{itemize}
where 
\begin{itemize}
    \item $\beta$ is a holomorphic section of $K_X \otimes \Lcal^{-2}$;
    \item locally if we write $\beta=\sigma e$ where $\sigma$ is a holomorphic $n$-form and $e$ is a unitary frame for $L$, then $\bar{\beta}=\bar{\sigma} e^*$ where $e^*$ denotes the metric dual of $e.$
    \item  $\Lcal=(L,\dbar_B)$.
\end{itemize}
In particular, $\int c_1(L)\wedge \omega^{n-1}>0$.
\end{prop}

\begin{proof}
Suppose $(A,\phi)$ is an $\SU(2)$-monopole. By definition, for any $1\neq t\in S^1$,  there exists an isomorphism $f: \Ecal \rightarrow \Ecal$ so that $f\circ \phi = t \phi \circ f.$ Since $b_1(f)$ and $b_2(f)$ are all constants, $f$ has constant eigenvalues. As Proposition \ref{prop2.16}, $f$ has two distinct eigenvalues $\lambda$ and $t\lambda$ and we can decompose 
$$
\Ecal=\Lcal \oplus \Lcal^{-1}
$$
where $$
\phi=\begin{pmatrix}
0&0\\
\beta&0
\end{pmatrix}
$$
for some $\beta\in H^0(X,\Lcal^{-2} \otimes K_X)$. By stability, $c_1(\Lcal).[\omega]^{n-1}>0.$ Given this, we can directly construct a Hermitian metric $H_{\Lcal}$ on $\Lcal$ with 
$$
F_{H_{\Lcal}} \wedge \frac{\omega^{n-1}}{(n-1)!} = \beta \wedge \bar{\beta}.
$$ 
In particular, the Chern connection of the induced metric $H_{\Lcal} \oplus H_{\Lcal^{-1}}$ on $\Ecal$ satisfies the Vafa-Witten equation. Then by uniqueness, the metric has to be a constant multiple of the Hermitian metric on the $\SU(2)$ bundle. The conclusion follows. 
\end{proof}

Summarizing the above, we have 
\begin{cor}\label{Monopoles}
The moduli space of $\SU(2)$ monopoles can be described as
$$\Mcal^{mon}=\{(\Lcal\oplus \Lcal^{-1}, \begin{pmatrix}
0&0\\
\beta&0
\end{pmatrix}): 0\neq \beta \in H^0(X, K_X \otimes \Lcal^{-2}), c_1(\Lcal).[\omega]^{n-1}> 0\}/\sim
$$
and $\Mcal^{mon}$ is compact. In particular, it is disjoint from the compactification of the moduli space of Hermitian-Yang-Mills connections.  
\end{cor}

\begin{proof}
The description of $M^{mon}$ follows directly from Lemma \ref{Lem7.2} and Proposition \ref{Prop7.3}. Since
$$
\int |\beta|^2\sim c_1(L).[\omega]^{n-1}\leq \frac{1}{2}c_1(K_X).[\omega]^{n-1},
$$
the space of such line bundles and holomorphic sections has to be compact. The second part follows from the fact that the Higgs fields of the monopoles are always nonzero. 
\end{proof}

\subsection{Stable trivial Higgs bundles}\label{Section8.2}
In this section,  we study the moduli space 
$$
\Mcal_0:=\{(\Ocal \oplus \Ocal, \phi) \text{ is stable }: \phi \in H^0(X, \End(\Ocal \oplus \Ocal) \otimes K_X)  \}/\sim.
$$
We first note
\begin{lem}
The $\CBbb^*$ invariant locus in $\Mcal_0$ is empty. 
\end{lem}
\begin{proof}
By Proposition \ref{Prop7.3}, there are no monopoles in $\Mcal_0$. In particular, the $\CBbb^*$ invariant locus must all have zero Higgs fields which is impossible by the stability condition.
\end{proof}

Below we fix a basis for $H^0(X, K_X)$ as $\{\sigma_i\}_{i=1}^m$, then by definition, we can always write 
$$
\phi=\sum_{i=1}^m M_i \sigma_i
$$
where $M_i$ are all rank two matrices. We have the following simple observation

\begin{lem}
$(\Ocal\oplus \Ocal, \phi)$ is stable if and only if $M_1, \cdots M_m$ have no common eigenvectors.
\end{lem}	

\begin{proof}
Suppose $v$ is a common eigenvector for $M_1, \cdots M_m$. Then consider 
$$
\Ocal \rightarrow \Ocal \oplus \Ocal, 1 \mapsto v
$$
which gives a rank one subsheaf of $\Ocal \oplus \Ocal$ invariant under $\phi$. This violates the stability of $(\Ocal \oplus \Ocal, \phi)$. Now assume $(\Ocal\oplus \Ocal, \phi)$ is not stable, i.e., there exists some rank one subsheaf $\Lcal\subset \Ocal \oplus \Ocal$ invariant under $\phi$ with 
$$
c_1(\Lcal).[\omega]^{n-1}=0.
$$
By considering $(\Ocal \oplus \Ocal)^* \rightarrow \Lcal^*$ which is nontrivial, we obtain $\Lcal^*$ has a section, thus $\Lcal^*$ must be trivial and so is $\Lcal$. Thus $\Lcal$ is generated by a constant vector $v$ invariant under $\phi$. It also satisfies
$$
\phi v = \sigma v 
$$
for some $\sigma\in H^0(X, K_X)$. Write $\sigma=\sum_i \lambda_i \sigma_i$, then 
$$
M_i v= \sigma_i v
$$
for all $i$, i.e., $\{M_i\}_{i=1}^m$ have a common eigenvector. 
\end{proof}

We have the following well-known linear algebra lemma
\begin{lem}
Two complex rank two  matrices $M$ and $N$ have no common eigenvectors if and only if 
$\det(MN-NM)\neq 0.$
\end{lem}

In particular, this gives 
\begin{cor}\label{Cor8.7}
$(\Ocal\oplus \Ocal, \phi=M \sigma_1+N \sigma_2)$ is stable if and only if 
$$
\det(MN-NM)\neq 0.
$$
In particular, $\Mcal_0$ is nonempty over compact K\"ahler manifolds with 
$$\dim (H^0(X, K_X))\geq 2.$$
\end{cor}

\bibliography{papers}

\providecommand{\MR}[1]{}
\providecommand{\bysame}{\leavevmode\hbox to3em{\hrulefill}\thinspace}
\providecommand{\MR}{\relax\ifhmode\unskip\space\fi MR }
\providecommand{\MRhref}[2]{%
  \href{http://www.ams.org/mathscinet-getitem?mr=#1}{#2}
}
\providecommand{\href}[2]{#2}
\begin{thebibliography}{10}

\bibitem{AG:03}
Luis {\'{A}}lvarez-C{\'{o}}nsul and Oscar Garc{\'{\i}}a-Prada,
  \emph{{H}itchin{\textendash}{K}obayashi correspondence, quivers, and
  vortices}, Communications in Mathematical Physics \textbf{238} (2003), no.~1,
  1--33.

\bibitem{BandoSiu:94}
Shigetoshi Bando and Yum-Tong Siu, \emph{Stable sheaves and
  {E}instein-{H}ermitian metrics}, Geometry and analysis on complex manifolds,
  World Sci. Publ., River Edge, NJ, 1994, pp.~39--50. \MR{1463962}

\bibitem{Besicovitch1931}
AS~Besicovitch, \emph{On sufficient conditions for a function to be analytic,
  and on behaviour of analytic functions in the neighbourhood of non-isolated
  singular points}, Proceedings of the London Mathematical Society \textbf{2}
  (1931), no.~1, 1--9.

\bibitem{ChenSun:19}
Xuemiao Chen and Song Sun, \emph{Reflexive sheaves, {H}ermitian-{Y}ang-{M}ills
  connections, and tangent cones}, Invent. Math. \textbf{225} (2021), no.~1,
  73--129. \MR{4270664}

\bibitem{Donaldson:87a}
Simon Donaldson, \emph{{Infinite determinants, stable bundles, and curvature}},
  Duke Math. J. \textbf{54} (1987), 231--247.

\bibitem{DonaldsonKronheimer:90}
Simon Donaldson and Peter Kronheimer, \emph{The geometry of four-manifolds},
  Oxford Mathematical Monographs, The Clarendon Press, Oxford University Press,
  New York, 1990, Oxford Science Publications. \MR{1079726}

\bibitem{DonaldsonThomas:98}
Simon Donaldson and Richard Thomas, \emph{Gauge theory in higher dimensions},
  The geometric universe ({O}xford, 1996), Oxford Univ. Press, Oxford, 1998,
  pp.~31--47. \MR{1634503}

\bibitem{GSTW:18}
Daniel Greb, Benjamin Sibley, Matei Toma, and Richard Wentworth, \emph{Complex
  algebraic compactifications of the moduli space of {H}ermitian
  {Y}ang--{M}ills connections on a projective manifold}, Geom. Topol.
  \textbf{25} (2021), no.~4, 1719--1818. \MR{4286363}

\bibitem{He20}
Siqi He, \emph{The behavior of sequences of solutions to the
  {H}itchin-{S}impson equations}, arXiv preprint arXiv:2002.08109 (2020).

\bibitem{Hitchin87}
Nigel Hitchin, \emph{The self-duality equations on a {R}iemann surface},
  Proceedings of the London Mathematical Society \textbf{3} (1987), no.~1,
  59--126.

\bibitem{HongTian:04}
Min-Chun Hong and Gang Tian, \emph{Asymptotical behaviour of the {Y}ang-{M}ills
  flow and singular {Y}ang-{M}ills connections}, Mathematische Annalen
  \textbf{330} (2004), 441--472.

\bibitem{Kobayashi:87}
Shoshichi Kobayashi, \emph{Differential geometry of complex vector bundles},
  Princeton University Press, 1987.

\bibitem{Lin:89}
Tzuemn-Renn Lin, \emph{The {H}ermitian-{Y}ang-{M}ills metrics and stability for
  holomorphic vector bundles with {H}iggs fields over compact {K}\"ahler
  manifolds}, University of California, San Diego, 1989.

\bibitem{LubkeTeleman:95}
Martin L\"{u}bke and Andrei Teleman, \emph{The {K}obayashi-{H}itchin
  correspondence}, World Scientific Publishing Co., Inc., River Edge, NJ, 1995.
  \MR{1370660}

\bibitem{LubkeTeleman:06}
\bysame, \emph{The universal {K}obayashi-{H}itchin correspondence on
  {H}ermitian manifolds}, Mem. Amer. Math. Soc. \textbf{183} (2006), no.~863,
  vi+97. \MR{2254074}

\bibitem{MMS:20}
Fernando Marchesano, Ruxandra Moraru, and Raffaele Savelli, \emph{A vanishing
  theorem for t-branes}, Journal of High Energy Physics \textbf{2020} (2020),
  no.~11, 1--32.

\bibitem{Mares:10}
Bernard~A Mares~Jr, \emph{Some analytic aspects of vafa-witten twisted n= 4
  supersymmetric yang-millseory theory}, Ph.D. thesis, Massachusetts Institute
  of Technology, 2010.

\bibitem{MSWW:16}
Rafe Mazzeo, Jan Swoboda, Hartmut Weiss, and Frederik Witt, \emph{Ends of the
  moduli space of higgs bundles}, Duke Mathematical Journal \textbf{165}
  (2016), no.~12, 2227--2271.

\bibitem{Mochizuki:16}
Takuro Mochizuki, \emph{Asymptotic behaviour of certain families of harmonic
  bundles on {R}iemann surfaces}, Journal of Topology \textbf{9} (2016), no.~4,
  1021--1073.

\bibitem{Nakajima:88}
Hiraku Nakajima, \emph{Compactness of the moduli space of {Y}ang-{M}ills
  connections in higher dimensions}, J. Math. Soc. Japan \textbf{40} (1988),
  no.~3, 383--392. \MR{945342}

\bibitem{Simpson:90}
Carlos Simpson, \emph{Harmonic bundles on noncompact curves}, Journal of the
  American Mathematical Society \textbf{3} (1990), no.~3, 713--770.

\bibitem{Simpson:92}
\bysame, \emph{Higgs bundles and local systems}, Inst. Hautes \'{E}tudes Sci.
  Publ. Math. (1992), no.~75, 5--95. \MR{1179076}

\bibitem{Tanaka:13}
Yuuji Tanaka, \emph{A weak compactness theorem of the {D}onaldson-{T}homas
  instantons on compact {K}{\"a}hler threefolds}, Journal of Mathematical
  Analysis and Applications \textbf{408} (2013), no.~1, 27--34.

\bibitem{ThomasTanaka17}
Yuuji Tanaka and Richard~P Thomas, \emph{Vafa-{W}itten invariants for
  projective surfaces ii: semistable case}, Pure and Applied Mathematics
  Quarterly \textbf{13} (2017), no.~3, 517--562.

\bibitem{ThomasTanaka19}
\bysame, \emph{Vafa-{W}itten invariants for projective surfaces i: stable
  case}, Journal of Algebraic Geometry (2019).

\bibitem{Taubes17}
Clifford~Henry Taubes, \emph{The behavior of sequences of solutions to the
  {V}afa-{W}itten equations}, arXiv preprint arXiv:1702.04610 (2017).

\bibitem{Tian:00}
Gang Tian, \emph{Gauge theory and calibrated geometry. {I}}, Ann. of Math. (2)
  \textbf{151} (2000), no.~1, 193--268. \MR{1745014}

\bibitem{UhlenbeckPreprint}
Karen Uhlenbeck, \emph{{A priori estimates for Yang-Mills fields}}, Unpublished
  manuscript.

\bibitem{Uhlenbeck:82a}
\bysame, \emph{Connections with {$L^{p}$} bounds on curvature}, Comm. Math.
  Phys. \textbf{83} (1982), no.~1, 31--42. \MR{648356}

\bibitem{UhlenbeckYau:86}
Karen Uhlenbeck and Shing-Tung Yau, \emph{{On the existence of
  Hermitian-Yang-Mills connections in stable vector bundles}}, Comm. Pure Appl.
  Math. \textbf{39} (1986), no.~2, S257--S293.

\bibitem{VafaWitten94}
Cumrun Vafa and Edward Witten, \emph{A strong coupling test of {S}-duality},
  Nuclear Physics B \textbf{431} (1994), no.~1-2, 3--77.

\end{thebibliography}
\end{document}